\documentclass[a4paper,12pt]{article}
\usepackage{amssymb}
\usepackage[reqno]{amsmath}
\usepackage{amsthm}

\usepackage{hyperref}
\usepackage{enumitem}

\newtheorem{theorem}{Theorem}
\newtheorem{lemma}{Lemma}
\newtheorem{corollary}{Corollary}
\newtheorem{proposition}{Proposition}
\theoremstyle{definition}

\newtheorem{remark}{Remark}

\newcommand{\IE}{\mathbb{E}}

\newcommand{\IN}{\mathbb{N}}

\newcommand{\IR}{\mathbb{R}}
\newcommand{\ZZ}{\mathbb{Z}}

\newcommand{\bsone}{\boldsymbol{1}}

\newcommand{\pr}[1]{\mathbb{P}\,[#1]}

\newcommand{\gs}{\succcurlyeq}
\newcommand{\ls}{\preccurlyeq}

\usepackage{color}

\DeclareMathOperator{\vol}{vol}
\DeclareMathOperator{\diam}{diam}
\DeclareMathOperator{\dist}{dist}

\makeatletter
\newcommand{\leqnomode}{\tagsleft@true}
\newcommand{\reqnomode}{\tagsleft@false}
\makeatother

\title{Random points are optimal\\ 
for the approximation of Sobolev functions} 
\author{David Krieg\footnote{Institut f\"ur Analysis, 
Johannes Kepler Universit\"at Linz, 4040 Linz, Austria.
\texttt{david.krieg@jku.at}, \texttt{mathias.sonnleitner@jku.at}.
} \
and Mathias Sonnleitner$^*$}

\date{}

\begin{document}

\maketitle

\begin{abstract}
We show that independent and uniformly distributed sampling points 
are asymptotically as good 
as optimal sampling points for the approximation of functions
from 
Sobolev spaces $W_p^s(\Omega)$ on bounded convex domains $\Omega\subset \IR^d$ in the $L_q$-norm if $q<p$.
More generally, we characterize the quality of arbitrary sampling point sets $P\subset \Omega$
via the $L_\gamma(\Omega)$-norm of the distance function $\dist(\cdot,P)$,
where $\gamma=s(1/q-1/p)^{-1}$ if $q<p$ and $\gamma=\infty$ if $q\ge p$.
This improves upon previous characterizations based on the covering radius of $P$.
\end{abstract}
\medskip

\centerline{\begin{minipage}[hc]{130mm}{
}
\end{minipage}}
\vspace{1cm}

Let $\Omega\subset \IR^d$ be a bounded convex domain. We study the problem of approximating
a function $f$ from the Sobolev space $W_p^s(\Omega)$
in the $L_q(\Omega)$-norm based on function values $f(x_j)$ 
on a finite set of sampling points $P=\{x_1,\hdots,x_n\}$.
This makes sense if $s > d/p$, 
in which case $W_p^s(\Omega)$ is compactly embedded
into the space of continuous functions $ C(\Omega)$.
The minimal worst-case error that can be achieved with 
the given sampling points 
is the number
\begin{equation} \label{eq:wce-app}
 e\big(P,W_p^s(\Omega)\hookrightarrow L_q(\Omega)\big) :=\,
 \inf_{S_P}\, \sup_{ \Vert f \Vert_{W_p^s(\Omega)} \le 1 } \Vert f - S_P(f) \Vert_{L_q(\Omega)},
\end{equation}
where the infimum is taken over all sampling operators of the form
\begin{equation}\label{eq:sampling_operator}
 S_P\colon W_p^s(\Omega) \to L_q(\Omega), 
 \qquad S_P(f)=\varphi\big(f(x_1),\hdots,f(x_n)\big).
\end{equation}
In general, we admit arbitrary mappings $\varphi\colon \IR^n \to L_q(\Omega)$, 
but sometimes it may be preferable to allow only linear mappings $\varphi$,
in which case we write $e^{\rm lin}$ instead of $e$.
We also study the related problem
of numerical integration on $\Omega$. To be precise, we consider
\begin{equation}\label{eq:wce-integration}
e\big(P,{\rm INT},W^s_p(\Omega)\big):=\inf_{S_P} \sup_{\|f\|_{W^s_p(\Omega)}\le 1} \left| \int_{\Omega} f(x){\rm d}x - S_P(f)\right|,
\end{equation}
where the infimum is now taken over all sampling operators of the form 
\begin{equation}\label{eq:sampling_operator_integration}
 S_P\colon W_p^s(\Omega) \to \mathbb{R},
 \qquad S_P(f)=\varphi\big(f(x_1),\hdots,f(x_n)\big).
\end{equation}
In this case, it does not matter whether we allow arbitrary
or only linear mappings $\varphi\colon\mathbb{R}^n\to\mathbb{R}$
since the infimum $e\big(P,{\rm INT},W^s_p(\Omega)\big)$ will be the same.
This is a classical result due to Smolyak and Bakhvalov,
see e.g.\  Theorem 4.7 in \cite{NW08}.

\medskip

There is a vast literature on the error of optimal sampling points. 
For example, it is known that the rate of convergence of this error is
\[
 e\big(n,W_p^s(\Omega)\hookrightarrow L_q(\Omega)\big) 
 := \inf_{\vert P \vert \le n} e\big(P,W_p^s(\Omega)\hookrightarrow L_q(\Omega)\big) 
 \,\asymp\, n^{-s/d+(1/p-1/q)_+}
\]
for the approximation problem and
\[
 e\big(n,{\rm INT},W^s_p(\Omega)\big) 
 := \inf_{\vert P \vert \le n} e\big(P,{\rm INT},W^s_p(\Omega)\big) 
 \,\asymp\, n^{-s/d}
\]
for the integration problem, where $a_+:=\max\{a,0\}$, $a\in\mathbb{R}$, 
and the infimum runs over all point sets $P\subset \Omega$ with at most $n$ points.
The same holds for the error $e^{\rm lin}$
of linear algorithms.
These are classical results for special domains like the cube,
see e.g.\ \cite[Chapter~3]{Cia78} and \cite[Chapter~6]{Hei94}. For general domains, we refer to Narcowich, Wendland and Ward \cite{NWW04} as well as Novak and Triebel \cite{NT06}.

In this paper, however, we are not so much interested 
in optimal sampling points. Although the question is quite interesting,
we feel that there are many applications where it is unrealistic
to assume that we can choose the sampling points at our convenience.
It might rather be realistic to assume that the 
points are independent random vectors
which are uniformly distributed on the domain.
That is, we get our data $f(x_j)$ for random parameters $x_j\in \Omega$
which are not under our control. This is also a typical assumption in learning theory and uncertainty quantification.  But is this a major drawback? Are random points significantly worse than optimal points?

\medskip

It can be understood from \cite{NWW04,NT06}
that the optimal order of convergence may be 
achieved with any set of sampling points that cover the domain well enough, meaning that the volume of the largest 
empty ball amidst the point set is of order $1/n$.
In other words, the covering radius,
which is the supremum of the distance function
\begin{equation*}
\dist(\cdot, P)\colon \mathbb{R}^d\to [0,\infty), 
\qquad \dist(x,P):=\min_{y\in P}\|x-y\|_2
\end{equation*}
over $x\in \Omega$, i.e., the radius of the largest empty ball,
has to be of order $n^{-1/d}$.
In this case, 
an optimal algorithm is given by moving least squares, see e.g.\  Wendland \cite{W04}.
In fact, it is quite common to use the covering radius, which is also called mesh norm, to bound errors of sampling-based algorithms,
see also~\cite{ALdST07, BDS+15, Duc78,EGO19,Mha10,WS93} for example.
Unfortunately, random point sets do not have optimal covering properties since the volume of the largest empty ball is of order $\log(n)/n$.
On the other hand, \emph{most} empty balls are of order $1/n$
and one might think that a few larger gaps do not matter
if the error is measured in the $L_q$-norm for sufficiently small $q$.
We use a local version of Wendland's result to 
show that this is indeed true for any $q<p$.
We obtain the following characterization
of the error of $L_q$-approximation and numerical integration. 

\leqnomode
\begin{theorem}\label{thm:main-result-sobolev}
	Let $\Omega\subset \mathbb{R}^d$ be a bounded convex domain,
	$1\le p,q \le \infty$ and $s\in\mathbb{N}$ with $s>d/p$. 
	Then we have for any nonempty and finite point set $P\subset \Omega$ the equivalences
	\begin{align}
	\tag{\textit a}
	&e^\ast\big(P,W^s_p(\Omega)\hookrightarrow L_q(\Omega)\big)\, \asymp \,
	\big\|\dist(\cdot, P)\big\|_{L_{\gamma}(\Omega)}^\alpha\\[9pt]
\tag{\textit b}
	 &e\big(P,{\rm INT},W^s_p(\Omega)\big)\, \asymp\, e^\ast\big(P,W^s_p(\Omega)\hookrightarrow L_1(\Omega)\big),
	\end{align} 
	where $e^\ast\in\{e, e^{\rm lin}\}$, the implied constants are independent of $P$, and
\[
 \gamma = \left\{ \begin{array}{ll} s(1/q-1/p)^{-1},  & q<p, \\ \infty, &  q\ge p, \end{array} \right.
 \qquad
 \alpha =  \left\{ \begin{array}{ll} s,  & q<p, \\ s-d(1/p-1/q), &  q\ge p. \end{array} \right.
\]
\end{theorem}
\reqnomode

\medskip

 This shows that the covering radius is not the right quantity
 to characterize the error in the case $q<p$,
 and in particular, for numerical integration.
 In general, the assumption 
 of a small covering radius is unnecessarily strong.
 Instead, we find that a sequence of point sets is asymptotically optimal
 if and only if the $\gamma$-norm of the distance function
 decays with the optimal rate $n^{-1/d}$.

\begin{corollary}\label{cor:characterization-sobolev}
	Let $\Omega\subset \mathbb{R}^d$ be a bounded convex domain,
   $1\le p,q \le \infty$ and $s\in\mathbb{N}$ with $s>d/p$.
	Assume that for each $ n\in\mathbb{N}$ an $ n $-point set $ P_n\subset \Omega $ is given. 
	These point sets are asymptotically optimal, i.e.,
	\[
	e^\ast\big(n,W^s_p(\Omega)\hookrightarrow L_q(\Omega)\big)
	\,\asymp \,e^\ast\big(P_n,W^s_p(\Omega)\hookrightarrow L_q(\Omega)\big),
	\]
	if and only if 
	\[
	\|\dist(\cdot, P_n)\|_{L_\gamma(\Omega)} \,\ls\, n^{-1/d},
	\]
	with $\gamma$ as in Theorem~\ref{thm:main-result-sobolev},
	$e^\ast\in\{e, e^{\rm lin}\}$
	and implied constants independent of $ n $.
\end{corollary}

 We emphasize that the main contribution of this paper is the characterization in the case $q< p$. The case $q\ge p$ is included for completeness. The results for $q<p$ seem to be novel already for $d=1$.
 To the authors, similar results only have been known
 for the spaces $W_\infty^s(\Omega)$ with $s\le 2$, 
 see Sukharev \cite{S79} and Pag\`es \cite{P98}. 

 \medskip
 
 By Theorem~\ref{thm:main-result-sobolev}\,(b), we clearly get the same characterization for numerical integration 
as for the problem of $L_1$-approximation.
In a certain sense, this characterization may serve as an asymptotic (and weighted) analogue 
in (isotropic) Sobolev spaces to the connection between errors of quasi-Monte Carlo rules in several other function spaces 
and various types of the geometric notion of discrepancy 
as surveyed for example in \cite[Section~9]{NW10}.

\begin{remark}[More general domains]\label{rem:domains}
In fact, the proof of Theorem~\ref{thm:main-result-sobolev} for $q\ge p$ works 
for all domains that satisfy 
an interior cone condition (as considered in \cite{NWW04}) and which admit a bounded linear extension operator 
${\rm ext}\colon W^s_p(\Omega) \to W^s_p(\IR^d)$ with
${\rm ext}(f)\vert_\Omega = f$ for all $f\in W^s_p(\Omega)$.
This includes all bounded Lipschitz domains (as considered in \cite{NT06}), 
see Lemmas~\ref{lem:Lip-implies-CC} and \ref{lem:extension-sobolev}.
Our proof for $q<p$ additionally requires
a strong local interior cone condition as described in Lemma~\ref{lem:local-CC},
which is implied by the assumed convexity of the domain.
However, we conjecture that also the case $q<p$
may be extended (at least) to all bounded Lipschitz domains.
\end{remark}

\begin{remark}[More general spaces]\label{rem:spaces}
Theorem \ref{thm:main-result-sobolev}
 may also be extended to more general function spaces,
 including Sobolev-Slobodeckij spaces of fractional smoothness,
Triebel-Lizorkin spaces, 
H\"older-Zygmund spaces and Bessel potential spaces,
see Section~\ref{sec:extensions}.
\end{remark}

\begin{remark}[Sobolev bounds on functions with scattered zeros]
Theorem~\ref{thm:main-result-sobolev} also gives the following bound: 
If $f\in W_p^s(\Omega)$ satisfies $f\vert_P=0$, then
\[
\Vert f \Vert_{L_q(\Omega)} \,\le\, c \,\big\|\dist(\cdot, P)\big\|_{L_{\gamma}(\Omega)}^\alpha \, \Vert f \Vert_{W_p^s(\Omega)},
\]
where $\gamma=\gamma(s,p,q)$ and $\alpha=\alpha(s,p,q,d)$ are as in Theorem~1
and $c$ is a constant independent of $f$ and $P$.
This improves upon the corresponding result in \cite[Theorem~1.1]{NWW04} for $q<p$
in the sense that the covering radius is replaced by the $\gamma$-norm of the distance function.
The paper \cite{NWW04} also gives upper bounds for
the Sobolev norms $\Vert f \Vert_{W_q^r(\Omega)}$ with $r<s-d/p$.
It would be interesting to see whether one can perform the same replacement in these bounds.
\end{remark}

\begin{remark}[Quantization]
 The quantity $\|\dist(\cdot,P)\|_{L_{\gamma}(\Omega)}^{\gamma}$ is studied extensively in the theory of quantization of measures, where it is known as quantization error or distortion. Asymptotics for the infimum over all $n$-point sets are known and the characterization in Corollary~\ref{cor:characterization-sobolev} is similar to the concept of asymptotically optimal quantizers. For more information we refer to Graf and Luschgy~\cite{GL00}.  	
\end{remark}
 
Our result enables us to study the quality of random sampling points which have almost optimal covering properties. In fact, we have the following proposition.
 
\begin{proposition}\label{pro:randomcovering}
Let $X_1,X_2,\ldots$ be independent and uniformly distributed random vectors on a bounded convex domain $\Omega\subset \mathbb{R}^d$ and let $0<\alpha<\infty$. Consider the random $n$-point set $P_n=\{X_1,\ldots,X_n\}$. Then
\[
\mathbb{E}\,\|\dist(\cdot,P_n)\|_{L_{\gamma}(\Omega)}^{\alpha}\,\asymp\, 
\begin{cases}
n^{-\alpha/d} & \text{if } 0<\gamma<\infty,\\
n^{-\alpha/d}(\log n)^{\alpha/d} & \text{if } \gamma=\infty. 
\end{cases}
\]
\end{proposition}

This immediately yields the following result on the quality of random sampling points.

\begin{corollary}\label{thm:intro}
Let $X_1,X_2,\ldots$ be independent and uniformly distributed on a bounded convex domain $\Omega\subset \mathbb{R}^d$, $1\le p,q \le \infty$ and $s\in\mathbb{N}$ with $s>d/p$. 
Consider the random $n$-point set $P_n=\{X_1,\ldots,X_n\}$. Then
\[
 \IE\, e^\ast\big(P_n,W_p^s(\Omega)\hookrightarrow L_q(\Omega)\big) 
 \,\asymp\, \begin{cases}
 	e^\ast\big(n/\log n,W^s_p(\Omega)\hookrightarrow L_q(\Omega)\big) & \text{if } q\ge p,\\
	e^\ast\big(n,W^s_p(\Omega)\hookrightarrow L_q(\Omega)\big) & \text{if } q< p.\vphantom{\Big|}
	\end{cases}
\]
\end{corollary}

This means that, in expectation, random points are asymptotically optimal
for $L_q$-approximation on $W^s_p(\Omega)$
whenever $q<p$.
In particular, random points are optimal for 
integration on $W^s_p(\Omega)$
for all $p>1$.
For an almost sure result, see Corollary~\ref{cor:as}. 
\medskip

All the above upper bounds 
 are achieved by a linear algorithm $S_P\colon W_p^s(\Omega)\to L_q(\Omega)$ 
 which only depends on the domain $\Omega$, the point set $P\subset \Omega$,
 and a smoothness threshold $m\in\IN$.
 It satisfies the bounds from Theorem~\ref{thm:main-result-sobolev} for all $1\le p,q \le \infty$ as well as all $s\in\IN$ with $s\le m$ and can be outlined as follows. 
 
 \medskip

\noindent\textbf{Algorithm.} Let $\Omega$ as in Theorem~\ref{thm:main-result-sobolev} and $m \in \IN$. 
\begin{itemize}
\item[a)] Given a finite and nonempty point set $P\subset \Omega$ construct a covering of $\Omega$ by good cubes $Q_1,\hdots,Q_N$ according to Proposition~\ref{prop:the-covering}.
\item[b)] Given the data $\big(x,f(x)\big)_{x\in P}$ for some unknown $f\in W_p^s(\Omega)$ and some $y\in \Omega\setminus P$, select an index $1\le i\le N$ with $y\in Q_i$ and compute $S_Pf(y)$ using the moving least squares method from \cite[Theorem~4.7]{W04} for $\overline{\Omega\cap Q_i}$, $P\cap Q_i$ and $m$.
\end{itemize}

In this paper, we do not discuss the implementation and the computational cost of these two steps.
Details on the second step may be found in~\cite{W04}.
In the paper \cite{KNS}, we propose a variant of this algorithm
where the computation of $S_Pf(x)$ at a certain point of interest $x\in \Omega$
does not require to compute the hole covering from the first step. 
We refer to Section~\ref{subsec:hilbert} for a more detailed description of optimal algorithms in the Hilbert space case via a reproducing kernel.

\medskip

For $q\ge p$, random points are not optimal.
We note, however, that we only lose a logarithmic factor
and that random information is still almost optimal. For the cube,
this was already observed in \cite{HKNPUsurvey}.
There, in Section~2.3, the case $q<p$ was stated as an open problem, which is resolved by Corollary~\ref{thm:intro}.
 Further recent results on the general question of the quality of random
information may be found in \cite{HKNPUellipsoids, HKNV20, KU19, Ull20}.
We point to the fact that there are also situations
where random information is much worse than
optimal information, see \cite[Section~6]{HKNV20} for an example.
\medskip

The remainder of the paper is organized as follows.
In Section~\ref{sec:preliminaries} we introduce notation
and present some basic facts on domains, polynomial reproduction
and Sobolev spaces.
Section~\ref{sec:proof} is dedicated to our geometric characterization
of the error and the proof of Theorem~\ref{thm:main-result-sobolev}.
We also discuss algorithms for numerical integration
and give a first example to illustrate the benefits of our characterization.
In Section~\ref{sec:random}, we apply the result to random sampling points 
for which a limit theorem is derived.
Finally, in Section~\ref{sec:extensions}, we discuss how our 
results may be extended
to more general function spaces.

\medskip

We want to note that there remain several unresolved issues. 
In particular, we encourage readers to 
(a)~examine the $d$-dependence of the constants in Theorem~~\ref{thm:main-result-sobolev}, 
(b)~examine and improve upon the computational cost of our algorithm, 
(c)~derive similar bounds for existing approximation algorithms, 
which currently are stated in terms of the covering radius,
(d)~obtain similar results for Besov spaces, and 
(e)~study more general domains.
Note that, while this paper was under revision,
we already extended our results to compact Riemannian manifolds in \cite{KSmani}.

\section{Preliminaries}
\label{sec:preliminaries}

Let us first fix some basic notation.
For $d\in\mathbb{N}$ and $0<p\leq \infty$, we write 
\[
\|x\|_p:=
\begin{cases}
\big(|x_1|^p+\cdots+|x_d|^p\big)^{1/p},& \text{if }0<p<\infty,\\
\max_{1\leq i\leq d}|x_i|,& \text{if }p=\infty
\end{cases}
\]
for the $\ell^p$-(quasi-)norm of  
$x=(x_1,\ldots,x_d)\in\mathbb{R}^d$. The space $\mathbb{R}^d$ will be equipped with the standard inner product $\langle \cdot,\cdot\rangle$.
We write $B_d^p(x,r):=\{y\in\mathbb{R}^d: \|x-y\|_p<r\}$ for 
the $\ell^p$-ball of radius $r>0$ centered at $x$. 
If $ c>0 $, we write $ cB_d^p(x,r) $ for the concentric $ \ell^p $-ball $ B_d^p(x,cr) $.
If $p=2$, we often omit the $p$ in these notations. 
We write $\mathbb{S}^{d-1}$ for the unit sphere, which is the 
boundary of $B_d^2(0,1)$. 
We denote the $ d $-dimensional Lebesgue-measure by $ \vol_d$ and frequently omit the dimension $ d $.
\medskip

We assume that all random vectors are defined on a common probability space $ (S,\Sigma,\mathbb{P}) $.
For a set $\Omega\subset \IR^d$ with finite and positive volume, 
an $\IR^d$-valued random variable $X$ will be called a uniformly distributed point in $\Omega$ 
if $\pr{X\in A}=\vol(A\cap \Omega)/\vol(\Omega)$ for all 
Lebesgue-measurable $A\subset\mathbb{R}^d$.
\medskip

The space of all continuous functions $f\colon\Omega\to\mathbb{R}$ on an open set 
$\Omega\subset\mathbb{R}^d$ is denoted by $C(\Omega)$. 
We write $f|_{B}$ for the restriction of $f\colon A\to\IR$ to $B\subset A\subset\mathbb{R}^d$. 
The space of polynomials on $\mathbb{R}^d$ 
of degree at most $m\in \mathbb{N}$ is written $\mathcal{P}_m^d$.
 If $0<p<\infty$, the function space $L_p(\Omega)$ is the collection of all 
(equivalence classes of) Lebesgue-measurable functions $f\colon\Omega\to\mathbb{R}$ 
with finite (quasi-)norm $\|f\|_{L_p(\Omega)}:=\big(\int_{\Omega}|f(x)|^p{\rm d}x\big)^{1/p}$. 
The space $L_{\infty}(\Omega)$ is the space of all essentially bounded functions $f\colon\Omega\to\mathbb{R}$
equipped with the norm $\|f\|_{L_{\infty}(\Omega)}:={\rm {esssup}}_{x\in \Omega}|f(x)|$. 
We use the convention that $a/0=\infty$ and $a/\infty=0$
as well as $\infty/a=\infty$ for all $a\in(0,\infty)$.
\medskip

For two nonnegative functions $a$ and $b$ 
defined on the same set, 
we will write $a\ls b$ whenever $a \le c b$ holds for a third function $c$
that only depends on a specific list of parameters.
Unless specified otherwise, 
this list consists of the domain $\Omega$, 
the dimension $d\in\IN$, the smoothness parameter $s>0$,
and the integrability parameters $p,\tau,q,\gamma,\gamma_1,\gamma_2\,{>0}$.
The function $c$ is called the implied constant.
We write $a\gs b$ if $b\ls a$
holds and $a\asymp b$
if both relations are satisfied.

\subsection{Domains and discrete sets}

In the following, let $\Omega\subset\mathbb{R}^d$ be a bounded domain (i.e., open and nonempty)
and let $P\subset \mathbb{R}^d$ be finite and nonempty. We are interested in the $L_\gamma(\Omega)$-norm of the distance function
\begin{equation*}
\dist(\cdot, P)\colon \mathbb{R}^d\to [0,\infty), 
\qquad \dist(x,P):=\min_{y\in P}\|x-y\|_2
\end{equation*}
for $0<\gamma \le \infty$. In particular, for $\gamma=\infty$, we obtain the covering radius
\[
h_{P,\Omega}:=\sup_{x\in\Omega}\dist(x,P)
=\|\dist(\cdot,P)\|_{L_{\infty}(\Omega)}
\]
of $ P $ with respect to $ \Omega $.
This is the minimal radius such that the (closures of the) balls $B(x,h_{P,\Omega}),$ where $ x\in P,$ cover $\Omega$.
Clearly, the finite volume of $\Omega$ gives $\|\dist(\cdot,P)\|_{L_{\gamma_1}(\Omega)}\ls \|\dist(\cdot,P)\|_{L_{\gamma_2}(\Omega)}$, 
whenever $0<\gamma_1\leq \gamma_2\le \infty$. Moreover, it is well known that we have
\begin{equation}
\label{eq:averagecoveringlower}
\inf_{|P| \le n} \|\dist(\cdot, P)\|_{L_{\gamma}(\Omega)} \,\asymp\, n^{-1/d} 
\qquad \text{for every} \quad
0<\gamma\leq \infty.
\end{equation}
The upper bound is quite obvious since the bounded set $\Omega$ may be covered
by $n$ balls of radius $\asymp n^{-1/d}$.
The lower bound is due to a standard volume argument:
We choose the constant $c$ such that the union of the $n$ balls 
with center in $P$ and radius $cn^{-1/d}$ has volume at most $\vol(\Omega)/2$,
which means that $\dist(\cdot,P)\ge c n^{-1/d}$ on a subset of $\Omega$ 
with volume at least $\vol(\Omega)/2$.

\medskip

A special Lipschitz domain in $\mathbb{R}^d$, $d\geq 2$, 
is the collection of all points $x=(x',x_d)$ with $x'\in\mathbb{R}^{d-1}$ such that
\[
h(x')<x_d<\infty,
\]
where $h\colon\mathbb{R}^{d-1}\to \IR$ is some Lipschitz function, i.e., 
there exists a constant $C>0$ with $|h(x')-h(y')|\leq C\|x'-y'\|_2$ for all $x',y'\in\mathbb{R}^{d-1}$. 
The set $\Omega$ is called a bounded Lipschitz domain (cf.\ \cite{NT06}) if there are points $x_1,\ldots,x_N\in\partial\Omega$ on the boundary and 
radii $r_1,\ldots,r_N>0$ such that $\partial\Omega$ is covered by the balls 
$B(x_1,r_1),\ldots,B(x_N,r_N)$ and
\[
B(x_i,r_i)\cap \Omega= B(x_i,r_i)\cap \Omega_i, \quad i=1,\ldots,N,
\]
where $\Omega_i$ is a suitable rotation of a special Lipschitz domain in $\mathbb{R}^d$.
\medskip

A (closed) cone with apex $x\in\mathbb{R}^d$, direction $\xi\in\mathbb{S}^{d-1}$, height (or radius) $r>0$ and opening angle $\theta\in (0,\pi/2)$ is given by
\[
C(x,\xi,r,\theta):=\left\{x+\lambda y\colon y\in \mathbb S^{d-1}, \langle y, \xi\rangle \geq \cos\theta,\lambda\in [0,r]\right\}.
\]
We will need Lemma 3.7 from \cite{W04} in the following special case.
\begin{lemma}\label{lem:ballincone}
Every cone $C(x,\xi,r,\theta)$ contains a ball of radius $c_{\theta}r$ with $c_{\theta}:=\frac{\sin \theta}{1+\sin\theta}$.
\end{lemma}
A general set $\Omega\subset \mathbb{R}^d$ is said to satisfy 
an interior cone condition (cf.\ \cite{W04}) if there are $r>0$ and $\theta\in (0,\pi/2)$ such that for all $x\in \Omega$ there is a unit vector $\xi(x)\in \mathbb{S}^{d-1}$ such that the cone
$C(x,\xi(x),r,\theta)$ is contained in $\Omega$. The following simple observations will be useful.
\begin{lemma}\label{lem:ballinball}
	Let $ \Omega\subset \mathbb{R}^d $ satisfy an interior cone condition with parameters $ r $ and $ \theta $. If $ B(x,\varrho)$ is a ball with center $ x\in\Omega $ and radius $ 0<\varrho\leq r $, there is a ball $ B(y,c_{\theta}\varrho)$ contained in $ B(x,\varrho)\cap \Omega$ with $c_{\theta} $ as in Lemma \ref{lem:ballincone}.
\end{lemma}
\begin{proof}
	By the cone condition, there is a cone with apex $x$, height $\varrho$ and angle $\theta$
	such that its interior is contained in $ B(x,\varrho) \cap \Omega$. Now Lemma \ref{lem:ballincone} completes the proof.
\end{proof}

Convex domains satisfy an interior cone condition, see e.g.\ \cite[Proposition~11.26]{W04}.

\begin{lemma}\label{lem:convex-cone}
Let $\Omega\subset \mathbb{R}^d$ be a bounded convex domain containing a ball of radius $r$. Then the closure of $\Omega$, denoted $\overline{\Omega}$, satisfies an interior cone condition with radius $r$ and angle $\theta=2\arcsin\bigl(r/2\diam(\Omega)\bigr)$.
\end{lemma}


Convex domains are also Lipschitz domains. A similar result may be found in Dekel and Leviatan \cite[Lemma 2.3]{DL04}.
\begin{lemma}\label{lem:convex-lip}
Every bounded convex domain is a bounded Lipschitz domain.
\end{lemma}

\begin{proof}
If $\Omega\subset \mathbb{R}^d$ is a bounded convex domain, 
we find a ball $B(x_0,r)\subset \Omega$. As $\partial\Omega$ is compact, 
there are points $x_1,\ldots, x_N$ such that the balls 
$B(x_i,r/2)$, $i\le N$, cover $\partial\Omega$. 
For all $i\le N$ we show that 
\begin{equation}\label{eq:Lip}
 B(x_i,r/2)\cap \Omega=B(x_i,r/2)\cap \Omega_i,
\end{equation} 
where $\Omega_i$ is a rotation of a special Lipschitz domain in $\mathbb{R}^d$. 
Applying a suitable rotation (and translation), 
we may assume that $x_0=0$ 
and that $x_i=(0,\hdots,0,a)$ for some $a\le -r$. Consider the open ball $B'=B_{d-1}(0,r)$ in $\IR^{d-1}$. 
For $x'\in B'$, we define the set $A(x') = \{x_d\in\IR \colon (x',x_d)\in \Omega\} $.
Since $\Omega$ is convex and open, $A(x')$ is an open interval. 
Moreover, $A(x')$ is nonempty since $0\in A(x')$.
We define $h(x')$ to be the infimum of $A(x')$.
The convexity of $\Omega$ implies that the function $h\colon B'\to \IR$ is convex.
Since every convex function on a convex domain in $\IR^{d-1}$ is Lipschitz
on every compact subset of the domain, see \cite{WSU72},
the function $h$ is Lipschitz on the closure of the ball $B^*=B_{d-1}(0,r/2)$.
This Lipschitz continuity carries over to the whole $\IR^{d-1}$
if we set $h(\lambda x')=h(x')$ for all $x' \in \partial B^*$ and $\lambda \ge 1$
(thereby redefining $h$ on $B'\setminus B^*$).
It remains to note that for every $x=(x',x_d)\in B(x_i,r/2)$ it holds that
\[
 x \in \Omega 
 \quad\Leftrightarrow\quad 
 x_d \in A(x')
 \quad\Leftrightarrow\quad 
 x_d > h(x'),
\]
proving \eqref{eq:Lip} for the special Lipschitz domain $\Omega_i=\{(x',x_d)\colon x_d>h(x')\}$.
\end{proof}

It seems to be well known that bounded Lipschitz domains
satisfy an interior cone condition.
This is used in \cite{NT06} and also stated by Adams and Fournier \cite[4.11]{AF03} for domains with the strong local Lipschitz condition, which include bounded Lipschitz domains.
Since we did not find it elsewhere, we include the proof.

\begin{lemma}\label{lem:Lip-implies-CC}
Every bounded Lipschitz domain satisfies an interior cone condition. 
\end{lemma}

\begin{proof}
We first show that special Lipschitz domains satisfy an interior cone condition.
Let therefore 
\[
 \Omega_0=\left\{ x=(x',x_d)\in \IR^d \colon x_d >h(x')  \right\}
\]
where $h\colon\mathbb{R}^{d-1}\to \IR$ is a Lipschitz continuous function and $d\ge 2$. We set 
\[
L=\sup_{x\ne y} \frac{\vert h(x)-h(y) \vert}{\Vert x - y \Vert},
\qquad \theta=\arctan L^{-1},
\qquad \xi=(0,\hdots,0,1)\in\IR^d
\]
and $\theta=\pi/2$ if $L=0$.
Let $x=(x',x_d)\in \Omega_0$ and $z=(z',z_d)\in C(x,\xi,\infty,\theta)$.
Then we have $z=x+\lambda y$ for some $\lambda \ge 0$ and $y\in\mathbb S^{d-1}$
with $y_d\ge \cos \theta$, and therefore,
\begin{multline*}
 z_d = x_d + \lambda y_d > h(x') + \lambda \cos \theta
 \ge h(z') - L \Vert z'-x' \Vert + \lambda \cos \theta \\
 = h(z') - L \lambda \sqrt{1-y_d^2} + \lambda \cos \theta
 \ge h(z') + \lambda( \cos \theta - L \sin \theta)
 = h(z').
\end{multline*}
We thus obtain that $C(x,\xi,\infty,\theta) \subset \Omega_0$ for all $x\in\Omega_0$.
\medskip

Let now $\Omega\subset \IR^d$ be a bounded Lipschitz domain.
Choose points $x_1,\ldots,x_N\in\partial\Omega$ and radii $r_1,\ldots,r_N>0$ such that 
$\partial\Omega$ is covered by the balls 
$B_i=B(x_i,r_i)$ and $B_i\cap \Omega= B_i\cap \Omega_i$, where $i\le N$ and $\Omega_i$ is a rotation of a special Lipschitz domain.
Clearly, the interior cone condition is preserved under rotations,
such that the $\Omega_i$ satisfy an interior cone condition
with angle $\theta_i$ (and infinite height).
First, we note that there is some $a\in(0,1)$ such that
the balls $aB_i$ still cover $\partial \Omega$:
Else, using the compactness of $\partial \Omega$, we would 
obtain a convergent sequence $(y_k)$ in $ \partial \Omega$ with
$\Vert y_k - x_i \Vert \ge (1-1/k) r_i$ for all $k\in \IN$ and $i\le N$.
Its limit $y\in\partial \Omega$ would satisfy $\Vert y - x_i \Vert \ge r_i$
and therefore $y\not\in B_i$, a contradiction.
\medskip

Let thus $b=(a+1)/2$,
$r=\min_i r_i$ and $\theta=\min_i \theta_i$, and let $x\in \Omega$.
If $x\in bB_i$ for some $i\le N$, then there is a cone with apex $x$, height $(1-b)r=(1-a)r/2$
and angle $\theta$ which is contained in both $B_i$ and $\Omega_i$
and therefore in $\Omega$.
On the other hand, if $x\not\in b B_i$ for all $i\le N$,
then the whole ball with center $x$ and radius $(b-a)r=(1-a)r/2$
is contained in $\Omega$: Else, there would be some $y\in\partial\Omega$
with $\Vert y-x \Vert \le (b-a)r$ and since $y\in a B_i$ for some $i\le N$
also $\Vert x_i-x \Vert < ar_i+ (b-a)r \le br_i$, which is a contradiction.
Thus $\Omega$ satisfies an interior cone condition with angle $\theta$
and height $(1-a)r/2$.
\end{proof}

A convex domain also satisfies the following local version of the cone condition.

\begin{lemma}
\label{lem:local-CC}
Let $\Omega\subset\mathbb{R}^d$ be a bounded convex domain containing a ball of radius $r>0$. Then the closure of the intersection of $\Omega$ with any cube $B_d^\infty(x, \varrho )$ centered at $x\in \overline{\Omega}$ and having $\ell^{\infty}$-radius $0< \varrho \le r$ satisfies an interior cone condition with height $c_\theta  \varrho $ and angle $\theta'=2\arcsin(c_{\theta}/4\sqrt d)$, where $c_\theta$ is as in Lemma~\ref{lem:ballincone} and $\theta$ is as in Lemma~\ref{lem:convex-cone}.
\end{lemma}

\begin{proof}
By Lemma~\ref{lem:convex-cone}, the set $\overline \Omega$ contains a cone with height $ \varrho \le r$, 
apex $x$ and angle $\theta$.
Clearly, this cone is contained in $\overline{\Omega \cap B(x, \varrho )}$,
which is a subset of $A(x, \varrho ):=\overline{\Omega\cap B_d^\infty(x, \varrho )}$.
By Lemma~\ref{lem:ballincone}, there is a ball $B(y,c_{\theta} \varrho )$ in this cone and thus in $A(x, \varrho )$. 
The proof is finished if we apply Lemma~\ref{lem:convex-cone}
to the convex set $A(x, \varrho )$ since its diameter is at most $2\varrho \sqrt{d}$.
\end{proof} 

One of our main tools will be the following result due to Wendland,
which we state (almost) in the formulation of 
Theorem 4.7 in his book \cite{W04}. 
It will provide us with a well-behaved linear algorithm. 

\begin{lemma}[\cite{W04}]\label{lem:wendland}
Let $K \subset\mathbb{R}^d$ be a compact set satisfying an interior cone condition with parameters $r$ and $\theta$, and let $m\in \mathbb{N}$. There 
are constants $c_0,c_1,c_2>0$ depending solely on $\theta,m$ and $d$ 
such that for any finite nonempty $P\subset K$ with covering radius
$h_{P,K}\le c_1 r$ there exist continuous 
functions $u_x\colon K\to\mathbb{R}$, $x\in P,$ with
\begin{enumerate}[label={\emph{(\roman*)}}]
\item $\displaystyle \pi(y)=\sum_{x\in P} \pi(x) u_x(y)$ for all $y\in K$
and $\pi\in \mathcal{P}_m^d $,
\item $\displaystyle \sum_{x\in P} |u_x(y)| \le c_0$ for all $y\in  K$ and
\item $u_x(y)=0$ for all $y\in K$ and $x\in P$ with $\Vert x-y \Vert \ge c_2 h_{P, K}$.
\end{enumerate}
\end{lemma}

\begin{proof}
 To derive this statement from \cite[Theorem~4.7]{W04},
 we will show that there is a subset $X$ of $P$ with $h_{X, K}\le 2 h_{P, K}$
 and separation distance
 \[
  q_X := \frac 12 \min_{x,y\in X, x\ne y} \Vert x - y \Vert_2 \ge \frac 12 h_{P, K}
 \]
 and apply Theorem~4.7 to the subset $X$ (setting $u_x=0$ for all $x\in P \setminus X$).
 To this end, we choose an arbitrary $x_1\in P$ and recursively choose $x_i \in P\setminus \bigcup_{j<i} B(x_j,h_{P, K})$ until the latter set is empty. We obtain a subset $X$ of $P$ which clearly satisfies $q_X\ge \frac 12 h_{P, K}$. Moreover, for any $y\in  K$, there is some $x\in P$ with $\Vert y- x \Vert_2 \le h_{P, K}$. Since $P\setminus \bigcup_{\tilde x \in X} B(\tilde x,h_{P, K})=\emptyset$,
 there is some $\tilde x \in X$ with $\Vert x - \tilde x \Vert_2 \le h_{P, K}$. The triangle inequality gives $\Vert y - \tilde x \Vert_2 \le 2 h_{P, K}$ and therefore $h_{X, K}\le 2 h_{P, K}$.
\end{proof}

\subsection{Sobolev spaces}
\label{sec:Sobolev-spaces}

Let $\Omega\subset \IR^d$ be a domain. 
For $ s\in\mathbb{N} $ and $ 1\leq p\le \infty $ we consider the Sobolev space
\[
W^s_p(\Omega):=\big\{f\in L_p(\Omega)\colon \|f\|_{W^s_p(\Omega)}:=\Big(\sum_{|\alpha|\leq s} \|D^{\alpha}f\|_{L_p(\Omega)}^p\Big)^{1/p}<\infty\big\},
\]
where $ \alpha\in \mathbb{N}_0^d $ is a multiindex, $ |\alpha|=\alpha_1+\ldots+\alpha_d $
and $D^\alpha f = \frac{\partial^{|\alpha|}}{\partial x_1^{\alpha_1}\cdots\partial x_d^{\alpha_d}} f$ denotes a weak partial derivative of order $|\alpha|$.
This is a Banach space with respect to the norm $\Vert \cdot \Vert_{W^s_p(\Omega)}$.
In addition to the norm, 
we will also work with the semi-norm
\[
 |f|_{W^s_p(\Omega)}:= \Big(\sum_{|\alpha|= s} \|D^{\alpha}f\|_{L_p(\Omega)}^p\Big)^{1/p},
\]
which has better scaling properties. If $\Omega$ is sufficiently regular, functions from the Sobolev space $W^s_p(\Omega)$ may be 
extended to $W^s_p(\IR^d)$, see Stein \cite[Section~VI.3]{Ste71}. 
Note that Stein uses the notion of minimally smooth domains which entails bounded Lipschitz domains.

\begin{lemma}[\cite{Ste71}]
\label{lem:extension-sobolev}
Let $\Omega \subset \IR^d$ be a bounded Lipschitz domain, 
$1\le p\le \infty$ and $s\in \IN$ with $s>d/p$.
Then there is a bounded linear operator 
${\rm ext}\colon W^s_p(\Omega) \to W^s_p(\IR^d)$ with
${\rm ext}(f)\vert_\Omega = f$ for all $f\in W^s_p(\Omega)$.
\end{lemma}

We will also use the following result about optimal polynomial approximation on cubes.
It easily follows from the Bramble-Hilbert lemma or generalized Poincar\'e inequality by scaling, see \cite[Lemma~4.3.8]{BS08} or \cite[Lemma\,1.1.11]{M85}.

\begin{lemma}\label{lem:polyapp-sobolev}
For any $1\le p\le \infty$ and $s\in \IN$ with $s>d/p$,
there is a constant $c_\ast>0$ 
such that the following holds. 
For any $0< \varrho \le 1$, any 
$\ell^\infty$-ball $Q$ of radius $ \varrho $ 
and any $f\in W^s_p(Q)$, there is a polynomial $\pi$ of degree at most $s$
such that
\[
 \sup_{x\in Q}\vert(f - \pi)(x)\vert \le c_\ast \,  \varrho ^{s-d/p} \vert f \vert_{W_p^s(Q)}.
\]
\end{lemma}

\begin{proof}
We use an affine map $T$ from $Q$ to $(-1,1)^d$, 
 on which we apply the continuous embedding of $W^s_p\bigl((-1,1)^d\bigr)$ into $C\bigl((-1,1)^d\bigr) $
 and the Bramble-Hilbert lemma
 to find constants $a,b>0$ 
 and a polynomial $\pi$ of degree at most $s$ with
 \begin{equation*}
    \sup_{y\in (-1,1)^d} \left|f\circ T^{-1}(y) - \pi(y)\right|
	\,\le\, a\, \|f\circ T^{-1} - \pi\|_{W^s_p((-1,1)^d)}
	\, \le \, a\,b\, |f\circ T^{-1}|_{W^s_p((-1,1)^d)}.
\end{equation*}
A change of variables gives
\[
    \sup_{x\in Q} \left|f(x) - \pi(Tx)\right|
    \,=\, \sup_{y\in (-1,1)^d} \left|f\circ T^{-1}(y) - \pi(y)\right|
\]
and
\[
	|f\circ T^{-1}|_{W^s_p((-1,1)^d)} \,=\, \varrho^{s-d/p} |f|_{W^s_p(Q)}.
\] 
\end{proof}

\section{The geometric characterization}
\label{sec:proof}

In this section, we give a proof of our main result, Theorem~\ref{thm:main-result-sobolev}. For the proof, let us fix $\Omega$ and $p,q,s$ as in the statement of Theorem~\ref{thm:main-result-sobolev}. 
All constants in this section are allowed to depend on 
these parameters. 
Let $ r \leq 1$ and $\theta\in (0,\pi/2)$ be such that $\Omega$ contains a ball of radius $r$ and  $\overline{\Omega}$
satisfies an interior cone condition with radius $ r $ and angle $\theta$, see Lemma~\ref{lem:convex-cone}. 
In Section~\ref{subsec:large-q}, we first prove the characterization in the simpler case $q\ge p$,
where only the largest hole amidst the point set matters.  The case $q<p$ is more involved and will be treated in Section~\ref{subsec:small-q}.  We then turn to numerical integration in Section~\ref{subsec:int}, discuss optimal algorithms for $p=2$ in Section~\ref{subsec:hilbert} and conclude with an example
in Section~\ref{subsec:example}.

\subsection{The case $q\ge p$}
\label{subsec:large-q}

Let us start with the upper bound.
We basically extend the proof of \cite[Proposition~21]{NT06}
to point sets with non-optimal covering radius. 
To this end, let $m>s$ be an integer and let $c_0,c_1,c_2>0$ be as in Lemma~\ref{lem:wendland}. 
We can assume that $ h_{P,\Omega}\le c_1  r $ as the upper bound is trivial if this does not hold.
We consider the linear algorithm
\[
 S_P\colon W_p^s(\Omega) \to L_q(\Omega), \quad S_P(f) = \sum_{x\in P} f(x) u_x,
\]
with $u_x$ provided by Lemma~\ref{lem:wendland} for the compact set $ \overline{\Omega} $.
The balls $B(y,h_{P,\Omega})$, $y\in \Omega$, 
cover the set $\overline{\Omega}$.
By compactness and Vitali's Lemma, 
there is a finite selection of pairwise disjoint balls 
$B_i=B(y_i,h_{P,\Omega})$ with $y_i\in \Omega$ 
 and $i\le N$ 
such that the balls $3B_i=B(y_i,3h_{P,\Omega})$ still cover $\overline{\Omega}$.
\medskip

Let $f\in W_p^s(\Omega)$ with $\Vert f \Vert_{W_p^s(\Omega)} \le 1$.
By Lemma~\ref{lem:extension-sobolev}, 
we may assume that $f\in W_p^s(\IR^d)$ with $\Vert f \Vert_{W_p^s(\IR^d)} \le c_3$.
Let $Q_i$ be the$\ell^\infty$-ball with center $y_i$ and radius $(3+c_2)h_{P,\Omega}$.
By Lemma~\ref{lem:polyapp-sobolev}, there are polynomials $\pi_i$ 
of degree at most $s$ such that
\begin{equation*}
 \sup_{y \in Q_i} \big\vert (f-\pi_i)(y)\big\vert \le c_4\, h_{P,\Omega}^{s-d/p} \vert f \vert_{W_p^s(Q_i)}.
\end{equation*}
For each $y\in \Omega_i:= 3B_i \cap \Omega$, we note that $u_x(y)=0$ for $x\not\in Q_i$, and obtain
\begin{align}
\nonumber
 \big\vert (f - S_P f)(y) \big\vert &= \big\vert (f - \pi_i)(y) - S_P(f-\pi_i)(y) \big\vert \\
 \label{eq:polyapp}
 &\le \big\vert (f- \pi_i) (y)\big\vert + \Big\vert \sum_{x\in P} (f-\pi_i)(x) u_x(y) \Big\vert 
 \le c_5\, h_{P,\Omega}^{s-d/p} \vert f \vert_{W_p^s(Q_i)},
\end{align}
where $c_5=(1+c_0) c_4$.
In particular, this yields
\[
 \Vert f - S_P(f) \Vert_{L_\infty(\Omega)} 
 \le c_5\,  h_{P,\Omega}^{s-d/p} \vert f \vert_{W_p^s(\IR^d)},
\]
which proves the case $q=\infty$.
For $p \le q<\infty$, we use $ \Omega_i\subset Q_i $ and \eqref{eq:polyapp} to get
\begin{multline*}
 \Vert f - S_P(f) \Vert_{L_q(\Omega)}^q 
 \le \sum_{i=1}^N \int_{\Omega_i} \big\vert \big(f - S_P(f)\big)(y) \big\vert^q {\rm d} y
 \le c_5^q\sum_{i=1}^N \, h_{P,\Omega}^{sq-dq/p} \vert f \vert_{W_p^s(Q_i)}^q \vol(\Omega_i) \\
 \le c_6\, h_{P,\Omega}^{sq-dq/p+d} \sum_{i=1}^N \vert f \vert_{W_p^s(Q_i)}^q
 \le c_6\, h_{P,\Omega}^{sq-dq/p+d} \Big(\sum_{i=1}^N \vert f \vert_{W_p^s(Q_i)}^p\Big)^{q/p}.
\end{multline*}
Note that, independently of $ P $, every $x\in \IR^d$ is contained in at most $ c_7\in \mathbb{N} $ of the $ N $ cubes $Q_i$ since 
the balls $B_i$ are disjoint and all have the same radius.
Therefore,
\begin{equation}\label{eq:norms-efficient}
 \sum_{i=1}^N \vert f \vert_{W_p^s(Q_i)}^p
 = \sum_{\vert \alpha \vert =s} \int_{\IR^d} \Big( \vert D^\alpha f(x) \vert^p \sum_{i=1}^N  \mathbf{1}_{Q_i}(x)\Big) {\rm d} x
 \le c_7\, \vert f \vert_{W_p^s(\IR^d)}^p \le c_7 c_3^p
\end{equation}
and we arrive at the desired inequality. 
\medskip
	
For the lower bound we use the well-known technique of \emph{fooling functions}:
We construct a function $f_\ast$ from the unit ball of $W_p^s(\Omega)$
which vanishes on $P$ and has a large $L_q$-norm.
Since any algorithm of the form~\eqref{eq:sampling_operator}
cannot distinguish $f_\ast$ from $-f_\ast$, i.e., it satisfies $S_P(f_*)=S_P(-f_*)$,
we have
	\begin{equation}\label{eq:fooling}
	\begin{split}
	 e\big(P,W_p^s(\Omega)&\hookrightarrow L_q(\Omega)\big) 
	 \,=\, \inf_{S_P} \sup_{\|f\|_{W^s_p(\Omega)}\le 1} \Vert S_P(f) - f \Vert_{L_q(\Omega)}\\
	 &\ge\, \inf_{u\in L_q(\Omega)}\, \max\left\{ \Vert u - f_* \Vert_{L_q(\Omega)}, \Vert u + f_* \Vert_{L_q(\Omega)} \right\}
	  \ge\, \|f_\ast\|_{L_q(\Omega)}.
\end{split}
\end{equation}
In the case $q\ge p$ it is enough to consider a function $f_\ast$
which is supported in the largest hole, compare \cite[Theorem~11]{HKNPUsurvey}
for $\Omega=[0,1]^d$. Let therefore $ x_0\in \Omega $ be such that
$ \dist(x_0,P)\ge h := \min\{ r ,\frac 12 h_{P,\Omega}\} $. 
Then the ball $ B(x_0,h)$ does not contain any point of $ P $. By Lemma \ref{lem:ballinball} we find a ball $B(y, \varrho)$ with $ \varrho := c_\theta h, $ which is contained in $ \Omega\cap B(x_0,h) $. 
 We take a smooth non-negative function $ \varphi $ supported in $ B(0,1) $ with $ \varphi(0)=1 $ and consider $ f(x)= \varphi\big(\varrho^{-1}(x-y)\big)$. 
One can easily compute the 
scaling properties
\begin{equation}\label{eq:scaling-LB}
	\|f\|_{L_q(\Omega)}\asymp \varrho^{d/q}\qquad 
	\text{and} \qquad 	\|f\|_{W^s_p(\Omega)} \ls \varrho^{-s+d/p},
\end{equation}
taking into account that $\varrho \le 1$ and $s>d/p$. This yields that $ f_*:=\|f\|_{W^s_p(\Omega)}^{-1} f$ satisfies 
	\[ 
	\|f_*\|_{W^s_p(\Omega)}\le 1, \quad 
	f_*|_P=0
  \quad\text{and}\quad
  \|f_*\|_{L_q(\Omega)} \gs h_{P,\Omega}^{s-d(1/p-1/q)}, 
	\]
	which completes the proof. 
	$\hfill \square$
	
\subsection{The case $q<p$}
\label{subsec:small-q}

We start with the upper bound,
using the following approximation property. 
As this useful result is somewhat technical, let us paraphrase it with ``Locally good point sets yield locally good approximations". 

\begin{lemma}
\label{lem:localestimate-extended}
Let $\Omega\subset \mathbb{R}^d$ be a bounded convex domain containing a ball of radius $r\in (0,1]$ and let $m\in\IN$.
There is a constant $c\in(0,1)$, depending solely on $m$ and $\Omega$ such that for any cube $Q=B_d^{\infty}(x,\varrho)\subset\IR^d$ of radius $0< \varrho \le  r $ centered at $x\in\overline{\Omega}$
 and any finite point set $P\subset \Omega$ with
\begin{equation}\label{assumption:good-cubes}
 \sup_{y\in \Omega\cap Q} \dist(y,P) \le c  \varrho ,
\end{equation}
there are continuous functions $u_x\colon \Omega\cap Q\to \IR$, $x\in P$, with the following property.
For any $s\in \{1,\hdots,m\}$ and $1\le p\le \infty$,
there is a constant $C>0$ depending only on $\Omega,s$ and $p$, such that, for all $f\in W^s_{p}(\IR^d)$, we have
\[
\sup_{y\in \Omega\cap Q} \Big\vert f(y) - \sum_{x\in P} f(x) u_x(y) \Big\vert
\le C  \varrho ^{s-d/p}|f|_{W^s_{p}(Q)}.
\]
\end{lemma}

\smallskip

\begin{proof}
By Lemma~\ref{lem:local-CC}, the set $\overline{\Omega\cap Q}$ satisfies an interior cone condition
with radius $ \varrho '=c_\theta  \varrho $ and angle $\theta'$ for some $c_{\theta}>0$. Let $c_0,c_1>0$ be as in Lemma~\ref{lem:wendland} for the parameters $\theta'$ and $m$. 
Note that $c_1$ can be chosen to be smaller than one. 
We set $c=c_{\theta} c_{\theta'} c_1/2$ with $c_{\theta'}$ as in Lemma~\ref{lem:ballincone}.
Then assumption \eqref{assumption:good-cubes} implies that
every ball $B(y,2c \varrho )$, $y\in Q\cap \Omega$, contains a point of $P$. 
By Lemma \ref{lem:ballinball} every ball $ B(x,2c_{\theta'}^{-1}c \varrho ), x\in Q\cap\Omega, 
$ contains a ball $ B(y,2c \varrho )\subset Q\cap \Omega $ 
and therefore a point of $P\cap Q$.
We thus have
\[
h_{P\cap Q,\Omega\cap Q} \leq 2c_{\theta'}^{-1} c  \varrho  = c_1  \varrho '.
\]
Thus, we may apply Lemma~\ref{lem:wendland} to the point set $P\cap Q$
within the compact set $\overline{\Omega\cap Q}$ and obtain
continuous functions $u_x\colon \Omega \cap Q \to \mathbb{R}$ for $x\in P$ with 
\[
 \sum_{x\in P} |u_x(y)|\leq c_0 \qquad\text{and}\qquad \pi(y)=\sum_{x\in P} \pi(x) u_x(y)
 \] 
for all $\pi\in\mathcal{P}_m^d$ and $y\in \Omega\cap Q$, where we set $u_x=0$ for $x \in P\setminus Q$.
For any $f\in W^s_{p}(\IR^d)$, 
by Lemma \ref{lem:polyapp-sobolev} there is a polynomial $\pi\in \mathcal{P}_m^d$ with
\[
\sup_{y\in Q} \big\vert (f-\pi)(y)\big\vert 
\,\le\, c_\ast  \varrho ^{s-d/p}|f|_{W^s_{p}(Q)}.
\]
Similar to Section \ref{subsec:large-q}, we get for all $y\in \Omega\cap Q$ that
\begin{multline*}
\Big\vert f(y) - \sum_{x\in P} f(x) u_x(y) \Big\vert
 = \Big\vert (f-\pi)(y) - \sum_{x\in P} (f-\pi)(x) u_{x}(y) \Big\vert\\
 \,\le\, \big\vert (f-\pi)(y) \big\vert + c_0\, \max_{x\in P\cap Q} \big\vert (f-\pi)(x) \big\vert
 \,\le\, (1+c_0)c_\ast\,  \varrho ^{s-d/p}|f|_{W^s_{p}(Q)},
\end{multline*}
as it was to be proven.
\end{proof}

In the following, we say that a cube $Q$
is a \emph{good cube} if it satisfies the conditions of Lemma~\ref{lem:localestimate-extended}. We now define good cubes $Q_P(x)$ for all $x\in\overline\Omega$. Without loss of generality, we may assume that
for any $x\in \overline{\Omega}$ 
there is some $ \varrho \in(0, r )$ 
such that
\begin{equation}\label{eq:goodradius}
 \sup_{y\in \Omega\cap B_d^\infty(x, \varrho )} \dist(y,P) < c  \varrho ,
\end{equation}
where $c$ is the constant from Lemma~\ref{lem:localestimate-extended}.
Otherwise, $\dist(\cdot,P)$ is larger than a constant
on a set of constant volume and the upper bound
of Theorem~\ref{thm:main-result-sobolev}\,(a) becomes trivial.
We define $r_P(x)$ to be the infimum 
over all $ \varrho \in(0, r )$ satisfying \eqref{eq:goodradius}
and let 
\begin{equation}\label{def:good-cube}
Q_P(x):=B_d^\infty\big(x,r_P(x)\big).
\end{equation}
Our definition ensures that $Q_P(x)$ is indeed a good cube. We will now show
that there is a covering of the domain by such cubes that is efficient in the sense that every point is 
covered only by a constant number of cubes.
For this, we need the following observation.

\begin{lemma}\label{lem:semi-cont}
The function $r_P\colon \overline{\Omega}\to (0,\infty)$ is upper semi-continuous.
\end{lemma}

\begin{proof}
Let $x\in\overline{\Omega}$. We first show $r_P(x)>0$. If $ \varrho >0$ is small enough,
then $B_d^\infty(x, \varrho )$ does not contain any point of $P$ except possibly $x$.
Since $c<1$, relation \eqref{eq:goodradius} cannot be satisfied for these values of $ \varrho $ and thus
$r_P(x)$ is positive.
To show the semi-continuity, let $\varepsilon>0$. Then there is some $ \varrho  \in [r_P(x),r_P(x)+\varepsilon]$ such that \eqref{eq:goodradius} is satisfied. Clearly, we can choose $\delta>0$ such that \eqref{eq:goodradius} is still satisfied if we replace the right hand side
by $ca$ with $a= \varrho  - \delta$.
For any $\tilde x\in \overline{\Omega}$ with $\Vert \tilde x - x \Vert_\infty<\delta$,
the cube $B_d^\infty(\tilde x,a)$ is contained in $B_d^\infty(x, \varrho )$.
We thus have
\[
 \sup_{y\in B_d^\infty(\tilde x,a) \cap \Omega} \dist(y,P)
 \le \sup_{y\in B_d^\infty(x, \varrho ) \cap \Omega} \dist(y,P) 
 < c a,
\]
which implies $r_P(\tilde x)\le a < r_P(x)+\varepsilon$.
\end{proof}

The semi-continuity is already enough to obtain the desired covering. Let us note that one can extract a suitable covering also from a Besicovitch-type covering result by de Guzman \cite[Theorem~1.1]{DeG75}.

\begin{lemma}\label{lem:efficient-covering}
 Let $K\subset\IR^d$ be compact and let $r\colon K\to(0,\infty)$
 be upper semi-continuous.
 If $K$ is nonempty, then there are points $y_1,\hdots,y_N\in K$ such that
 \begin{enumerate}[label={\emph{(\roman*)}}]
  \item\label{property:covering} The cubes $Q_i=B_d^\infty\big(y_i,r(y_i)\big)$, $1\le i \le N$, cover $K$.
  \item\label{property:disjoint} The cubes $Q_i/2$ are pairwise disjoint.
  \item\label{property:efficiency} Every $y\in \IR^d$ is contained in at most $2^d$ of the cubes $Q_i$.
 \end{enumerate}
\end{lemma}

\begin{proof}
 By compactness of $K$, 
 we can choose $y_1$ as a maximizer of $r_P$ on $K$.
 We set $r_1=r(y_1)$ and $Q_1=B_d^\infty\bigl(y_1,r(y_1)\bigr)$.
 Recursively, we can choose $y_i$ as a maximizer of $r_P$ 
 on $K\setminus \bigcup_{j<i} Q_j$ (which is also compact)
 and set $r_i=r(y_i)$ and $Q_i=B_d^\infty(y_i,r_i)$,
 as long as the set difference is nonempty.
 We first show that the sets $Q_i^\prime=B_d^\infty(y_i,r_i/2)$ 
 are pairwise disjoint. Assume that $y\in Q_i^\prime\cap Q_j^\prime$
 for some $j<i$. Then we have
 \[
  \Vert y_j - y_i \Vert_\infty \le \Vert y_j - y \Vert_\infty + \Vert y - y_i \Vert_\infty
  < r_j/2 + r_i/2
  \le r_j,
 \]
 since the numbers $r_i$ are nonincreasing,
 and thus $y_i \in Q_j$. A contradiction.
 
 We now show that our procedure terminates after finitely many steps.
 Assume for a contradiction that we obtain a whole sequence $(y_i)_{i\in\IN}$.
 Then we must have $r_i\to 0$, since otherwise we would have infinitely many
 disjoint cubes $Q_i^\prime$
 with volume larger than a constant. There is a subsequence $(y_{i_k})_{k\in\IN}$
 that converges to some $y_0\in K$. By assumption, $r_0=r(y_0)>0$. Since $r_i\to 0$, there is some $k\in \IN$ with $r_k<r_0$.
 As $r_k$ is a maximum of $r$ on $K\setminus\bigcup_{\ell<k} Q_\ell$,
 the point $y_0$ must be contained in $Q_\ell$ for some $\ell < k$. Since $Q_\ell$ is open and $y_{i_k}\to y_0$,
 we get that $y_{i_k} \in Q_\ell$ for all sufficiently large $k$. But by our construction, we have
 $y_i\not\in Q_\ell$ for all $i>\ell$. A contradiction.
 Thus, the procedure terminates and we obtain a finite point set 
 $y_1,\hdots,y_N$, satisfying \ref{property:covering} and \ref{property:disjoint}.
 
 To prove the efficiency of the covering, 
 let now $y\in \IR^d$ with $y\in Q_i \cap Q_j$ for some $j<i$.
 Assume that the signs of $y_i-y$ and $y_j-y$ are the same
 (i.e., $y_i$ and $y_j$ are in the same orthant with respect to $y$).
 Then we have
 \[
  \Vert y_j - y_i \Vert_\infty
  = \Vert (y_j - y) - (y_i -y) \Vert_\infty
  \le \max\left\{\Vert y_j - y \Vert_\infty, \Vert y_i - y \Vert_\infty \right\}
  < r_j
 \]
 and thus $y_i \in Q_j$, a contradiction.
 This means that the signs of all $y_j-y$ for which $y\in Q_j$
 must be different -- and there are at most $2^d$ different signs. 
\end{proof}

Lemma~\ref{lem:semi-cont} and Lemma~\ref{lem:efficient-covering} immediately yield an efficient covering of $\overline \Omega$ by good cubes.
It has the following properties.

\begin{proposition}\label{prop:the-covering}
Let $\Omega\subset \mathbb{R}^d$ be a bounded convex domain containing a ball of radius $r\in (0,1]$ and let $m\in\IN$.
For any finite and nonempty point set $ P\subset \Omega $, 
there are $y_1,\hdots,y_N\in\overline \Omega$ 
such that the good cubes $Q_i=Q_P(y_i)$ with radii $r_i=r_P(y_i)$ from \eqref{def:good-cube} cover $\overline \Omega$ and every $y\in \IR	^d$ is contained in at most $2^d$ of these cubes.
Each of the cubes $Q_i$ contains a ball $B_i$ of radius $d_i \asymp r_i$
such that these balls are pairwise disjoint subsets of $\Omega$ and empty of $P$.
\end{proposition}

\begin{proof}
With Lemma~\ref{lem:semi-cont} and Lemma~\ref{lem:efficient-covering}
we immediately get an efficient covering by good cubes $Q_i$
such that the cubes $Q_i/2$ are disjoint.
By the minimality of $r_i$, the set $Q_i/4\cap\Omega$ contains 
a point $z_i$ with $\dist(z_i,P) \ge c r_i/8$, where 
$c$ is as in Lemma~\ref{lem:localestimate-extended}. 
By Lemma~\ref{lem:ballinball}, we obtain a ball $B_i$ with radius 
$c_\theta c r_i/8$ that is contained in $\Omega\cap B(z_i,cr_i/8)$
and therefore does not contain a point of $P$. 
In particular, the balls $B_i$ are contained in $Q_i/2$ and are therefore pairwise disjoint.
\end{proof}

We are ready to complete the proof of the upper bound of Theorem~\ref{thm:main-result-sobolev}\,(a) 
for $q< p$.

\begin{proof}
Let $m\ge s$ and $y_i$, $r_i$, $d_i$, $Q_i$ and $B_i$ for $1\le i \le N$ 
 as in Proposition~\ref{prop:the-covering}.
 Let
 \[
 \Omega_i=(Q_i\cap\Omega)\setminus \bigcup_{j<i} \Omega_j
 \]
 for all $1 \le i\le N$. That is, $\Omega$ is the disjoint union
 of the sets $\Omega_i\subset Q_i$.
 For $x\in P$ and $y\in \Omega_i$, we define $u_x(y)$
 according to Lemma~\ref{lem:localestimate-extended},
 applied to the cube $Q_i$. 
 This yields bounded functions $u_x\colon \Omega\to \IR$
 and a linear algorithm 
\begin{equation}\label{eq:final-alg}
  S_P\colon W^s_{p}(\Omega) \to L_q(\Omega), \quad S_P(f)=\sum_{x\in P} f(x) u_x.
 \end{equation} 
 Let 
 $f\in W^s_{p}(\Omega)$ with
 $\Vert f\Vert_{W^s_{p}(\Omega)}\le 1$.
 Using Lemma~\ref{lem:extension-sobolev} we may assume that $f\in W^s_{p}(\IR^d)$ with
 $\|f\|_{W^s_{p}(\IR^d)}\le c_3$.
 By Lemma~\ref{lem:localestimate-extended}
 we have
 \begin{equation*}
 \Vert f - S_P(f) \Vert_{L_q(\Omega_i)}^q 
 \le c_4  r_i^{(s-d/p)q+d} |f|_{W^s_{p}(Q_i)}^q, 
 \end{equation*}
 where $(s-d/p)q+d=(\gamma+d)(1-q/p)$. With the obvious modification for $p=\infty$, H\"older's inequality yields
 \begin{equation*} 
 \Vert f - S_P(f) \Vert_{L_q(\Omega)}^q 
 \ls\, \sum_{i=1}^N r_i^{(\gamma+d)(1-q/p)}|f|_{W^s_{p}(Q_i)}^q
 \,\le\, \bigg(\sum_{i=1}^N r_i^{\gamma+d}\bigg)^{1-q/p} 
 \bigg(\sum_{i=1}^N |f|_{W^s_{p}(Q_i)}^p \bigg)^{q/p}.
 \end{equation*} 
 Since the cubes $Q_i$ form an efficient covering,
 we can proceed as in \eqref{eq:norms-efficient} 
 to obtain
 that the second factor in the previous estimate is bounded by a constant. Since the balls $B_i$ of radius $d_i\asymp r_i$ 
 are pairwise disjoint subsets of $\Omega\setminus P$, 
 we get $\dist(\cdot,P)\ge d_i/2$ on $B_i/2$ and
\[
 \sum_{i=1}^{N} r_i^{\gamma+d}
 \,\ls\, \sum_{i=1}^{N} r_i^{\gamma} \vol(B_i)
 \,\ls\,\sum_{i=1}^{N} \int_{B_i/2} \dist(x,P)^{\gamma}\,{\rm d}y
 \,\le\, \int_{\Omega} \dist(x,P)^{\gamma}\,{\rm d}y,
\]
which completes the proof, noting that $(1-q/p)=qs/\gamma$.
\end{proof}

This shows that the algorithm from \eqref{eq:final-alg},
which only depends on $\Omega$, $P$, and $m$,
achieves the error bound from Theorem~\ref{thm:main-result-sobolev} for all $q<p$ and $s\le m$.  Recall that the optimal bound in the case $q\ge p$ is 
achieved by the moving least squares method, see Section~\ref{subsec:large-q}. 
However, it is a nice universality property of the algorithm \eqref{eq:final-alg}
that it is also optimal in the case $q\ge p$ and thus, for all values of $p$ and $q$ and $s\le m$.

\medskip

We briefly explain how the upper bound for $p\ge q$ may be obtained. 
Using the notation from the proof of the case $q<p$, it holds that $r_i=r_P(y_i)\lesssim h_{P,\Omega}$, for all $1\le i\le N$, which follows from \eqref{eq:goodradius} and \eqref{def:good-cube}. Further, Lemma~\ref{lem:localestimate-extended} gives 
\[
\sup_{y\in\Omega_i}\big|(f-S_Pf)(y)\big| \,\lesssim\, h_{P,\Omega}^{s-d/p}\,|f|_{W^s_p(Q_i)}.
\]
Replacing \eqref{eq:polyapp} with this, one can proceed as in the proof in the case $q\ge p$ in Section~\ref{subsec:large-q} to obtain
\[
 \Vert f - S_P f \Vert_{L_q(\Omega)} \,\lesssim\, h_{P,\Omega}^{s-d/p+d/q}\, |f|_{W^s_p(\Omega)}. 
\]

\medskip


We will now prove the lower bound of Theorem~\ref{thm:main-result-sobolev}\,(a) for $q< p$. 
Again we construct a fooling function
with large $L_q$-norm that vanishes on the point set.
However, unlike in the case $q\ge p$, the fooling function 
has to be supported in multiple holes. 

\begin{proof}
Let $\varphi\colon\mathbb{R}^d\to\mathbb{R}$  be a smooth non-negative function supported within $B(0,1)$ with $\varphi(0)=1$. We choose $y_i$, $r_i$, $Q_i$ and 
$B_i=B(z_i,d_i)$ for $1\le i \le N$ as in Proposition~\ref{prop:the-covering}. For every $ i\le N $ and $x\in\IR^d$ define 
\[
\varphi_i(x):=d_i^{s+\gamma/p}\varphi\big(d_i^{-1}(x-z_i)\big).
\]
With this, $\varphi_i$ is supported in $B(z_i,d_i)$. With a substitution one can see that
\[
\|\varphi_i\|_{L_q(\mathbb{R}^d)}^q=d_i^{\gamma+d}\|\varphi\|_{L_q(\mathbb{R}^d)}^q
\]
since $q(s+\gamma/p)=\gamma$.
With another substitution we obtain for $p<\infty$ that
\begin{equation}\label{eq:anormphii}
\|\varphi_i\|_{W^s_{p}(\mathbb{R}^d)}^p\ls d_i^{\gamma+d} \|\varphi\|_{W^s_{p}(\mathbb{R}^d)}^p.
\end{equation}
Let $f:=\sum_{i=1}^N \varphi_i$. 
Note that $f\vert_P=0$.
Since the $\varphi_i$'s have pairwise disjoint support, we have
\begin{equation}\label{eq:rnormfp}
\|f\|_{L_q(\Omega)}^q=\sum_{i=1}^N \|\varphi_i\|_{L_q(\mathbb{R}^d)}^q=\sum_{i=1}^N d_i^{\gamma+d}\|\varphi\|_{L_q(\mathbb{R}^d)}^q.
\end{equation}
For the same reason we obtain from \eqref{eq:anormphii} that
\begin{equation}\label{eq:anormfp}
\|f\|_{W^s_{p}(\Omega)}^p\asymp \sum_{i=1}^N \|\varphi_i\|_{W^s_{p}(\mathbb{R}^d)}^p \ls \sum_{i=1}^N d_i^{\gamma+d}\|\varphi\|_{W^s_{p}(\mathbb{R}^d)}^p.
\end{equation}
If $p=\infty$, we replace \eqref{eq:anormfp} by
\[
\|f\|_{W^s_{\infty}(\Omega)}\asymp \max_{1\le i\le N} \|\varphi_i\|_{W^s_{\infty}(\mathbb{R}^d)}\ls \|\varphi\|_{W^s_{\infty}(\mathbb{R}^d)}
\]
and carry on analogously. The estimates \eqref{eq:rnormfp} and \eqref{eq:anormfp} imply for the normalized function 
$ f_*:=\|f\|_{W^s_{p}(\Omega)}^{-1} f$ with $ \|f_*\|_{W^s_{p}(\Omega)}=1 $ that 
\[
\|f_*\|_{L_q(\Omega)}=\frac{\|f\|_{L_q(\Omega)}}{\|f\|_{W^s_{p}(\Omega)}}\gs \left(\sum_{i=1}^N d_i^{\gamma+d}\right)^{1/q-1/p}\frac{\|\varphi\|_{L_q(\mathbb{R}^d)}}{\|\varphi\|_{W^s_{p}(\mathbb{R}^d)}}.
\]
Since the cubes $Q_i$ cover $\Omega$ (and all contain a point of $P$), we get
\begin{equation}\label{eq:LB-discretisation}
 \sum_{i=1}^{N} d_i^{\gamma+d}
 \,\gs\, \sum_{i=1}^{N} d_i^{\gamma} \vol(Q_i)
 \,\gs\, \sum_{i=1}^{N} \int_{Q_i} \dist(y,P)^{\gamma}{\rm d}y
 \,\gs\, \int_{\Omega} \dist(y,P)^{\gamma}{\rm d}y
\end{equation}
and thus
\begin{equation}\label{eq:lower}
\|f_*\|_{L_q(\Omega)}\gs \left(\int_{\Omega} \dist(x,P)^{\gamma}{\rm d}x\right)^{s/\gamma},
\end{equation}
where the implied constant also depends on the choice of $ \varphi $. 
As any algorithm of the form \eqref{eq:sampling_operator} cannot distinguish $f_*$
from $-f_*$, we obtain the desired lower bound using \eqref{eq:fooling}.
\end{proof}

\subsection{Integration}
\label{subsec:int}

In this section we prove the result of Theorem~\ref{thm:main-result-sobolev}\,(b) for numerical integration and give an example of an (asymptotically) optimal algorithm. 

\begin{proof}
For the upper bound, we simply consider the linear algorithm
\[
S_P^*: W^s_p(\Omega)\to \mathbb{R},\quad f\mapsto S_P^*(f)=\int_{\Omega} S_P(f)(x)\,{\rm d}x
\]
with $S_P(f)\in L_q(\Omega)\subset L_1(\Omega)$ as in the previous sections and observe that
\[
 \Big|\int_\Omega f(x)~{\rm d} x - S_P^*(f)\Big| \le \|f-S_P(f)\|_{L_1(\Omega)}.
\]
In order to prove the lower bound, it suffices to note that the fooling functions $f_{\ast}$ in 
Section~\ref{subsec:large-q} and Section~\ref{subsec:small-q} are nonnegative
and thus $\int f_\ast:=\int_{\Omega} f_\ast(x){\rm d}x= \|f_{\ast}\|_{L_1(\Omega)}$.
Since any algorithm of the form \eqref{eq:sampling_operator_integration} 
satisfies $S_P(f_*)=S_P(-f_*)$, we get
\begin{equation*}
	\begin{split}
	 e\big(P,{\rm INT},W_p^s(\Omega)\big) 
	 \,&=\, \inf_{S_P} \sup_{\|f\|_{W^s_p(\Omega)}\le 1} \big\vert S_P(f) -  {\textstyle \int} f \big\vert\\
	 &\ge\, \inf_{u\in\IR}\, \max\left\{ \big\vert u - {\textstyle \int} f_* \big\vert, \big\vert u + {\textstyle \int} f_* \big\vert \right\}
	  \ge\, {\textstyle \int} f_\ast
\end{split}
\end{equation*}
and thus the same lower bound as in \eqref{eq:lower} for $q=1$.
\end{proof}

\subsection{Special algorithms in the Hilbert case}
\label{subsec:hilbert}

In particular, the bounds from Theorem \ref{thm:main-result-sobolev} are valid for the optimal algorithm using the given data, i.e., the algorithm attaining the infimum in \eqref{eq:wce-app} or \eqref{eq:wce-integration}. For $p=2$, $W^s_2(\Omega)$ is a reproducing kernel Hilbert space (RKHS) and this algorithm can be described explicitly via the kernel. We refer to Aronszajn \cite{A50} for a definition and properties of RKHSs.

\medskip

To any RKHS $(H,\langle \cdot,\cdot\rangle_H)$ on $\Omega$ there corresponds a symmetric and positive definite mapping, a kernel, $K\colon\Omega\times\Omega\to\mathbb{R}$ such that if $x\in\Omega$, we have $K(x,\cdot)\in H$ and the reproducing property $f(x)=\langle f,K(x,\cdot)\rangle_H$ holds for all $ f\in H$. A theorem due to Riesz yields that every continuous linear functional $\ell\colon H\to\mathbb{R}$ can be represented by a function $h_{\ell}\in H$. In fact, one has $ h_{\ell}(x)=\langle h_{\ell},K(x,\cdot)\rangle_H=\ell(K(x,\cdot))$ for every $x\in\Omega$. 
\medskip

To be more concrete, set $H=W^s_2(\Omega)$, and denote the corresponding kernel by $K_s$. Because of the embedding $W^s_2(\Omega)\hookrightarrow C(\Omega)$ the integration functional $\int\colon W^s_2(\Omega)\to \mathbb{R}$ is continuous, and, by the above, it is represented by the function $h(\cdot)=\int_{\Omega} K_s(x,\cdot)\,{\rm d}x\in W^s_2(\Omega)$. For a point set $P=\{x_1,\ldots,x_n\}\subset\Omega$ and a linear algorithm $S_P\colon f\mapsto \sum_{i=1}^n w_i f(x_i)$ with weights $w=(w_1,\ldots,w_n)\in\mathbb{R}^n$ a well-known expression for the worst-case error of $S_P$ on $W^s_2(\Omega)$ is available (see e.g.\  (9.31) in \cite{NW10}).
It reads
\begin{equation}\label{eq:errorformula}
\sup_{\|f\|_{W^s_2(\Omega)}\le 1}\Bigl|\int_{\Omega}f(x){\rm d}x -S_P(f)\Bigr|^2= \langle w,\mathbf{K}w\rangle-2\langle w,b\rangle+c,
\end{equation}
where the matrix $\mathbf{K}:=\bigl(K_s(x_i,x_j)\bigr)_{i,j=1}^n$ is positive definite, $b:=\bigl(\int_{\Omega}K_s(x,x_i){\rm d}x\bigr)_{i=1}^n$, and $c:=\int_{\Omega}\int_{\Omega} K_s(x,y){\rm d}x {\rm d}y$. Finding the optimal weights $w^{\ast}=(w_1^{\ast},\ldots, w_n^{\ast})$ and thus the optimal (linear) algorithm given $P$ is easy in this case. As a function of the weights, the expression \eqref{eq:errorformula} is convex and therefore optimal weights are given by the unique solution of $\mathbf{K}w^{\ast}=b$.
Therefore, the knowledge of the kernel $K_s$ permits a computation of (an approximation of) the optimal (linear) algorithm for numerical integration in Theorem~\ref{thm:main-result-sobolev}. 

\medskip

For $L_q$-approximation, $1\le q\le \infty$, consider the algorithm 
\[
f \,\mapsto\, \sum_{j=1}^{n}w^{\ast}_j K_s(\cdot,x_j),\quad f\in W^s_2(\Omega),
\]
where $w^{\ast}$ is the unique solution of $\mathbf{K}w=\big(f(x_i)\big)_{i=1}^n$. This is the interpolant with minimal $W^s_2(\Omega)$-norm and satifies the error bounds from Theorem~\ref{thm:main-result-sobolev} as it attains the infimum in \eqref{eq:wce-app}. This is known and follows for example from Theorems 13.1 and 13.5 in \cite{W04}. Note that this kind of algorithm is used in radial basis function interpolation, see, e.g. \cite{W04} and the references therein. 

\medskip

For obtaining the kernel $K_s$, we equip the Sobolev space $W^s_2(\Omega)$ with the equivalent norm $\|f\|^{\ast}_{W^s_2(\Omega)}:=\inf\big\{ \|g\|_{W^s_2(\mathbb{R}^d)}\colon g|_{\Omega}=f\big\}$,
 where the equivalence is due Lemma \ref{lem:extension-sobolev}.
Then, Theorem~5 in \cite{A50} yields that the reproducing kernel $K_s$ of $W^s_2(\Omega)$ is the restriction of the reproducing kernel $K_{d,s}$ of $W^s_2(\mathbb{R}^d)$ to $\Omega\times\Omega$, that is $K_s(x,y)=K_{d,s}(x,y)$ for $x,y\in\Omega$.  An explicit form of $K_{d,s}$ can be found in Theorem 1 of Novak, Ullrich, Wo\'zniakowski and Zhang \cite{NUWZ18}, which states that
\begin{equation*}
K_{d,s}(x,y)=\int_{\mathbb{R}^d} \frac{\exp\big(2\pi i\langle x-y,u\rangle\big)}{1+\sum_{0<|\alpha|\le s}\prod_{j=1}^d (2\pi u_j)^{2\alpha_j}}\,{\rm d}u,\quad \text{for }x,y\in\mathbb{R}^d.
\end{equation*}

\subsection{A first example}
\label{subsec:example}

We illustrate the advantage of the characterization of optimal points in Corollary~\ref{cor:characterization-sobolev}
compared to conditions involving the covering radius. To this end, consider $n$-point sets $P_n$ on a bounded convex domain $\Omega$. How large can the largest hole admist $P_n$ be for this to be still optimal as a sampling set for $L_q$-approximation of $W^s_p(\Omega)$-functions? 
\medskip

Assume that for each $n$ the ball $B_n:=B(y_n,r_n)$ with $y_n\in\Omega$ and $r_n>0$ does not contain any points of $P_n$ and that the points of $P_n$ cover $\Omega\setminus B_n$ nicely, i.e., 
the covering radius of $P_n$ in $\Omega\setminus B_n$ is of order $n^{-1/d}$. Then we have
\[
\int_{\Omega}\dist(x,P_n)^{\gamma}~{\rm d}x \,\ls\, n^{-\gamma/d}+r_n^{\gamma+d}\quad \text{for all }0<\gamma<\infty.
\]
Rearranging, this means that
\begin{equation}\label{eq:optimallargesthole}
r_n\ls n^{-1/d+1/(\gamma+d)}\quad \text{implies}\quad \|\dist(\cdot,P_n)\|_{L_{\gamma}(\Omega)}\ls n^{-1/d}.
\end{equation}
This condition is also necessary since $\dist(\cdot,P_n) \gs r_n$ on $\Omega \cap B(x_n,\frac{r_n}{2})$,
which is of volume $\asymp r_n^d$ due to the interior cone condition, and thus
\[
\int_{\Omega}\dist(x,P_n)^{\gamma}~{\rm d}x \,\gs\, r_n^{\gamma+d}.
\]
Letting $1\le q<p \le \infty$ and $\gamma$ as above,
we obtain from Corollary \ref{cor:characterization-sobolev} 
that $(P_n)$ is asymptotically optimal for $L_q$-approximation
on $W_p^s(\Omega)$ if and only if
\begin{equation}
\label{eq:condition-largest-hole}
r_n\ls n^{-1/d+1/(\gamma+d)}.
\end{equation}
This means that the radius of the largest hole is allowed to exceed the optimal covering radius $n^{-1/d}$
by the polynomial factor $n^{1/(\gamma+d)}$. 
We also refer to \cite[Theorem 1.2]{BDS+15}, where the necessity of condition~\eqref{eq:condition-largest-hole} has been observed for numerical integration on the sphere.

\section{The optimality of random points}
\label{sec:random}

In the following we give the proof of Proposition \ref{pro:randomcovering}. First, we derive a strong asymptotic result from a theorem due to Cohort \cite{Coh04}, which he obtained in the context of random quantizers. Related work may be found in \cite[Section 9]{GL00} and \cite{Yuk08}.

\begin{proposition}\label{pro:limittheorem}
Let $0<\gamma<\infty$ and $X_1,X_2,\ldots$ 
be independent and uniformly distributed on a bounded domain $\Omega\subset \mathbb{R}^d$ which satisfies an interior cone condition. Consider the random $n$-point set $P_n=\{X_1,\ldots, X_n\}$. Then we havefor all $0<p<\infty$ that 
\[
n^{\gamma/d}\frac{1}{\vol(\Omega)}\int_{\Omega}\dist(x,P_n)^{\gamma}{\rm d}x
\,\xrightarrow{\text{a.s.\,and in } L_p}\, 
\left(\frac{\vol(\Omega)}{\vol\bigl(B(0,1)\bigr)}\right)^{\gamma/d}\Gamma\Bigl(1+\frac{\gamma}{d}\Bigr).
\]
\end{proposition}
\begin{proof}
We apply Theorem~1 and Theorem~2 in \cite{Coh04} and 
let $\mu=\nu$ be the uniform distribution on $\Omega$, i.e., 
these measures have densities $f=g=\frac{1}{\vol(\Omega)}\bsone_{\Omega}$, with $\bsone_{\Omega}$ being the indicator function of $\Omega$.  We check the assumptions T1.1-2 and T2.1-3 of the aforementioned theorems from \cite{Coh04}.
As Cohort remarks in Section 2, assumption T2.1 is satisfied for bounded probability density functions. 
Clearly, T1.2 and T2.3 are satisfied as well.
For checking assumptions T1.1 and T2.2, we note that the interior cone condition yields a constant $c_{\Omega}>0$ such that for every $x\in\Omega$ and every $0<\varrho\le 1+\sup_{x\in\Omega}\|x\|_2$ we have $\vol\bigl(B(x,\varrho)\cap\Omega\bigr)\ge c_{\Omega} \vol\bigl(B(x,\varrho)\bigr)$. 
Since $\Omega$ is bounded, assumptions T1.1 and T2.2 are validated 
and Theorems 1 and 2 can be applied.
Clearly, since convergence in $L_p$ implies convergence in $L_q$ for all $q<p$, the condition $p\in\IN\setminus\{1\}$ in \cite[Theorem~1]{Coh04} is obsolete.
\end{proof}

Regarding the asymptotic constant, let us note that if $\Omega$ is a centrally symmetric convex body and $B(0,1)$ is the maximal volume ellipsoid inside $\Omega$, the quantity $\bigl(\vol(\Omega)/\vol\bigl(B(0,1)\bigr)\bigr)^{1/d}$ is known as the volume ratio of $\Omega$, which plays an important role in Banach space geometry, see e.g.\  Szarek and Tomczak-Jaegermann \cite{ST-J80}.  Further, as $\gamma\to\infty$ the quantity $\bigl(\Gamma(1+\frac{\gamma}{d})\bigr)^{1/\gamma}$ tends to infinity, which is compatible with the fact that the $L_{\infty}$-norm of $\dist(\cdot,P_n)$ is typically of larger order than $n^{-1/d}$.

\begin{proof}[Proof of Proposition \ref{pro:randomcovering}]
Let $0<\gamma<\infty$.  By Lemma \ref{lem:convex-cone} a bounded convex domain satisfies an interior cone condition.  We set $p=\alpha/\gamma$ and obtain from Proposition~\ref{pro:limittheorem} that $n^{\gamma/d} \cdot \|\dist(\cdot,P_n)\|_{L_{\gamma}(\Omega)}^{\gamma}$ converges in $L_p$ to a positive constant, which implies that also $n^{\alpha/d} \cdot \mathbb{E}\, \|\dist(\cdot,P_n)\|_{L_{\gamma}(\Omega)}^{\alpha}$ converges to a positive constant.  This completes the proof.

\medskip

The case $\gamma=\infty$ is well known. 
It is strongly connected to the so-called coupon collector's problem,
which asks for the number of coupons (points) that a collector has to draw
in order to obtain a complete collection (hit every set in a diameter-bounded equal volume partition of $\Omega$).
Using that $\Omega$ satisfies an interior cone condition,
the stated result e.g.\ follows from Theorems~2.1 and 2.2 
in Reznikov and Saff \cite{RS16} with $\Phi(r)=r^d$.
\end{proof}

\begin{remark}
Proposition \ref{pro:limittheorem} extends to probability distributions different from the uniform distribution. Under certain conditions, it also allows for the sampling distribution to differ from the distribution with respect to which the $L_{\gamma}$-norm of $\dist(\cdot,P_n)$ is computed. 
\end{remark}

Proposition~\ref{pro:randomcovering} and Theorem~\ref{thm:main-result-sobolev}
immediately yield Corollary~\ref{thm:intro}
on the quality of random sampling points in expectation.
By Markov's inequality, it is clear that the upper bounds in Corollary~\ref{thm:intro} 
also hold with high probability.
In addition, we also obtain an almost sure result on the quality of random sampling points.
\begin{corollary}\label{cor:as}
Let $X_1,X_2,\ldots$ be independent and uniformly distributed on a bounded convex domain $\Omega\subset \mathbb{R}^d$, $1\le q<p \le \infty$ and $s\in\mathbb{N}$ with $s>d/p$. 
Consider the random $n$-point sets $P_n=\{X_1,\ldots,X_n\}$ for $n\in\IN$. 
Then, the following holds almost surely:
\[
 e^{\rm lin}\big(P_n,W_p^s(\Omega)\hookrightarrow L_q(\Omega)\big) 
 \,\asymp\, n^{-s/d}. 
\]
\end{corollary}
\begin{proof}
The upper bound is an immediate consequence of Theorem~\ref{thm:main-result-sobolev} and Proposition~\ref{pro:limittheorem}.
The lower bound holds for arbitrary $n$-point sets.
\end{proof}

\section{Extensions to Triebel-Lizorkin spaces}
\label{sec:extensions}

In this section we discuss how to extend 
Theorem~\ref{thm:main-result-sobolev} to a wider range of function spaces.
To that end, we consider the Triebel-Lizorkin spaces $F^s_{p\tau}(\Omega)$ for $0<p<\infty$, $0<\tau\le \infty$ and $s>d/p$ 
as defined e.g.\  in \cite{NT06}. We refrain from reproducing the definition and simply note that this family covers a
variety of interesting spaces:
\begin{itemize}
\item For $\tau=2$ and $1<p<\infty$, we obtain the fractional Sobolev space (or Bessel potential space) $H^s_p(\Omega)$, see e.g.\  \cite{NT06}. If additionally $s\in\mathbb{N}$, we arrive at
the classical Sobolev spaces $W^s_p(\Omega)$ as defined in Section~\ref{sec:preliminaries}.
\item For $s\not\in\IN$, $1\le p<\infty$ and $\tau=p$, we obtain the Sobolev-Slobodeckij space $W^s_p(\Omega)$ of fractional smoothness, see e.g.\  \cite{DS93} and note that $ F^s_{pp}(\Omega) $ is the Besov space $B^s_{pp}(\Omega)$.
\end{itemize}
We also want to discuss the 
H\"older spaces $C^s(\Omega)$, 
which are not included in this scale.
For $s\in\IN$, the H\"older space $C^s(\Omega)$ is the space
of all $s$ times continuously differentiable functions with
\[
  \Vert f \Vert_{C^s(\Omega)} := \max_{|\alpha| \le s}\, \sup_{x\in\Omega}\, \vert D^{\alpha}f (x) \vert < \infty.
\]
For $ s\not\in\IN $, it is defined as the space of functions $f \in C^{\lfloor s\rfloor}(\Omega)$ with
\begin{equation}\label{semi-norm-hoelder}
 |f|_{C^s(\Omega)}:=\max_{|\alpha|=\lfloor s\rfloor}\,\sup_{x\ne y}\, \frac{|D^{\alpha}f(x)-D^{\alpha}f(y)|}{\|x-y\|^{\{s\}}} < \infty,
\end{equation}
where $ s=\lfloor s\rfloor + \{s\} $ with $ 0< \{s\}<1 $.
It is equipped with the norm $\|\cdot\|_{C^{\lfloor s\rfloor}(\Omega)}+|\cdot|_{C^s(\Omega)}$.

Note that in this article equality between function spaces is meant up to equivalence of norms and the equivalence constants vanish in the asymptotic notation. Further, all mentioned spaces are continuously embedded in $ C(\Omega) $.
We obtain the following extension of our results on $L_q$-approximation.

\begin{theorem}\label{thm:main-result-extended}
Theorem~\ref{thm:main-result-sobolev}, Corollary~\ref{cor:characterization-sobolev} and Corollary~\ref{thm:intro} 
remain valid if we admit arbitrary real parameters $0<p,q,\tau\le \infty$ and $s>d/p$ 
and replace $W_p^s(\Omega)$ either by $C^s(\Omega)$ for $p=\infty$ or 
by $F^s_{p\tau}(\Omega)$ for $p<\infty$.
\end{theorem}

For the proof, let us first note that
the analogues of Corollary~\ref{cor:characterization-sobolev} and Corollary~\ref{thm:intro}
immediately follow from the analogue of Theorem~\ref{thm:main-result-sobolev}\,(a) 
if we employ \eqref{eq:averagecoveringlower} and Proposition~\ref{pro:randomcovering}.
Moreover, the extension of Theorem~\ref{thm:main-result-sobolev}\,(a) to the case $C^s(\Omega)$ for $s\in\IN$ is already included in Section~\ref{sec:proof}.
Namely, the upper bound is immediate from the continuous embedding 
$C^s(\Omega) \hookrightarrow W_\infty^s(\Omega)$.
The lower bound holds since our fooling functions $f_\ast$ 
for $W_\infty^s(\Omega)$ are smooth and non-negative.
For the remaining cases, 
we follow the lines of the proof of Theorem~\ref{thm:main-result-sobolev}.
We only discuss the necessary changes.

\begin{proof}[Proof of the upper bounds]
It suffices to consider the case $\tau=\infty$,
since $ F^s_{p\tau_1}(\Omega)\hookrightarrow F^s_{p\tau_2}(\Omega) $ for $ \tau_1\le \tau_2 $.
We replace $| f |_{W_p^s(\Omega)}$ by the following semi-(quasi-)norms:
\begin{itemize}
\item For $p=\infty$ and $s\not\in\IN$, we use the semi-norm $|f|_{C^s(\Omega)}$ as defined in \eqref{semi-norm-hoelder}.
\item For $p<\infty$, we use a semi-(quasi-)norm which is defined via the means 
\[
(d_t^{M,\Omega}f)(x) :=
t^{-d}\int_{V^M_{\Omega}(x,t)}\bigl|(\Delta_{h,\Omega}^M f)(x)\bigr|\,\text{d} h,
\]
where $M:=\lfloor s+1 \rfloor$, $\Delta_{h,\Omega}^M$ 
is an $M^{\rm{th}}$-order difference operator 
restricted to $\Omega$ and $V^M_{\Omega}(x,t)$ is 
the set of directions $h\in\IR^d$ of length less than $t>0$ with $x+a h\in\Omega$ for all $0\le a \le M$.
That is, $(d_t^{M,\Omega}f)(x)$ is an averaged mean of the $M^{\rm{th}}$-order differences of $f$ around $x$.
Putting
\[
|f|_{F^s_{p\infty}(\Omega)} :=
 \bigg\| \sup_{0\le t \le 1}  \frac{(d_t^{M,\Omega}f)(\cdot)}{t^s} \bigg\|_{L_p(\Omega)},
\]
it is known that we then have the equality 
 $F^s_{p\infty}(\Omega)=\{f\in L_{\infty}(\Omega)\colon |f|_{F^s_{p\infty}(\Omega)}<\infty\}$ and
$
\|\cdot \|_{L_{\max\{p,1\}}(\Omega)}+|\cdot|_{F^s_{p\infty}(\Omega)}
$
is an equivalent quasi-norm, see Proposition 6 in \cite{NT06} and set $ u=1$ as well as $r=\infty $.
The same is true for $\Omega=\IR^d$. 
\end{itemize}
It is readily verified that these semi-(quasi-)norms have
the following scaling property.
If $T\colon\IR^d\to\IR^d$ is of the form $T(x)= \varrho ^{-1} x + x_0$ 
with $ \varrho \le 1$ and $x_0\in\IR^d$ 
and $f\in F^s_{p\infty}(\Omega)$ or $f\in C^s(\Omega)$, then $g=f\circ T^{-1}$ satisfies
\begin{equation}
\label{eq:scaling}
 |g|_{F^s_{p\infty}(T\Omega)} \le  \varrho ^{s-d/p} |f|_{F^s_{p\infty}(\Omega)},
 \qquad
 |g|_{C^s(T\Omega)} =  \varrho ^{s} |f|_{C^s(\Omega)},\quad  \text{respectively.}
\end{equation}
Lemma~\ref{lem:extension-sobolev} holds without changes 
for the spaces $F^s_{p\tau}(\Omega)$ and $C^s(\Omega)$, see Rychkov \cite{Rych99},
where we note that $C^s(\Omega)=B^s_{\infty\infty}(\Omega)$ for $s\not\in\IN$.
If we use \cite[Corollary~11]{NT06} (for $p<\infty$) 
and \cite[Theorem~6.1]{DS80} (for $p=\infty$) instead of \cite[Lemma~1.1.11]{M85}
and the scaling properties \eqref{eq:scaling},
we see that Lemma~\ref{lem:polyapp-sobolev} concerning polynomial approximation on cubes 
remains valid under the modifications of Theorem~\ref{thm:main-result-extended}.
Thus, also Lemma~\ref{lem:localestimate-extended} defies our modifications.
It only remains to note that the semi-norm $|\cdot |_{F^s_{p\infty}(\Omega)}$ behaves equally well with respect to an efficient covering. Analogous to \eqref{eq:norms-efficient}, we have
\begin{equation}\label{eq:norms-efficient-fspace}
\sum_{j=1}^N |f|^p_{F^s_{p\infty}(Q_j)}
= \int_{\mathbb{R}^d} \left(\sup_{0\le t \le 1}  \frac{(d_t^{M,Q_j}f)(x)}{t^s}\right)^p
\sum_{j=1}^N \mathbf{1}_{Q_j}(x)\,{\rm d}x
\le c_7 |f|^p_{F^s_{p\infty}(\mathbb{R}^d)}
\le c_8
\end{equation}
since $d_t^{M,Q_j}f(x)\le d_t^{M,\mathbb{R}^d}f(x)$ for every $x\in\mathbb{R}^d$.
With these preparations at hand, we may copy the proofs of Section~\ref{sec:proof}. 
\end{proof}

\begin{proof}[Proof of the lower bounds]
Also the lower bound for the cases $q\ge p$ and $q<p=\infty$ can be copied.
The scaling properties \eqref{eq:scaling-LB} for $F^s_{p\tau}$ may be obtained
from Proposition~2.3.1/1 in \cite{ET96} and the translation-invariance of the quasi-norm.
The lower bound in the case $q<p<\infty$ requires different ideas 
as we cannot establish $\left\|\sum_{i} \varphi_i\right\|_{F^s_{p\tau}(\mathbb{R}^d)}^p\ls \sum_i \|\varphi_i\|^p_{F^s_{p\tau}(\mathbb{R}^d)}$ for general disjointly supported $\varphi_i$,
as used in \eqref{eq:anormfp} for the Sobolev spaces. 
For optimal point sets, this problem does not exist as we may consider bumps $\varphi_i$ which are supported in balls of the same radius which enables the use of localization methods, see \cite{NT06}.
For general point sets, a remedy can be provided with the help of atomic decompositions of the spaces $F^s_{p\tau}(\Omega)$ as treated by Triebel~\cite{Tri11}. 
\medskip

Let $y_i$, $d_i$ and $B_i$ be as in Proposition~\ref{prop:the-covering}. Let $K>s$ be an integer and let $Q$ be the cube of sidelength $1$ which is centered at the origin. 
Let $\psi\in C^{\infty}(\IR^d)$ be non-negative with $\mathrm{supp}\, \psi\subset Q$
such that $\psi(0)>0$ and $\Vert D^\alpha \psi \Vert_\infty \le 1$ for all $\alpha\in\IN_0^d$ with $|\alpha|\le K$. The exact choice of the function $\psi$ is not important for our argument. According to Definition 13.3 in \cite{Tri11} the function 
\[
\psi_{\nu m}(\cdot)
=2^{-\nu(s-d/p)}\psi(2^{\nu}\cdot-m)\quad \text{where }\nu\in \IN_0, \ m\in \ZZ^d,
\]
is an $(s,p)_{K,-1}$-atom supported in $Q_{\nu m}$, the cube with center $2^{-\nu}m$ and sidelength $2^{-\nu}$. We choose $m_i\in \ZZ^d$ and $\nu_i\in \IN_0$ with $2^{-\nu_i}\asymp d_i$ 
such that the dyadic cube $Q_i^*:=Q_{\nu_i m_i}$ is contained in $B_i$. Set
\[
f:=\sum_{i=1}^{N}\lambda_i \psi_i,\quad \text{where }\lambda_i:=2^{-\nu_i \beta} \text{ with }\beta=(\gamma+d)/p.
\]
This function is supported in $\Omega$ and satisfies $f|_P=0$. Also, $\|f\|_{F^s_{p\tau}(\Omega)}=\|f\|_{F^s_{p\tau}(\IR^d)}$ and by \cite[Theorem~13.8]{Tri11},
\[
\|f\|_{F^s_{p\tau}(\IR^d)}^p
\lesssim \|\lambda\|_{f_{p\tau}}^p
=\int_{\IR^d}\Big(\sum_{i=1}^N |\lambda_i 2^{\nu_i d/p}\mathbf{1}_{Q_i^*}(x)|^{\tau}\Big)^{p/\tau}{\rm d} x,
\]
where the implicit constant does not depend on the balls $B_i$ and $f_{p\tau}$ is a quasi-normed space of sequences.
For every $x\in\IR^d$ only one summand is not equal to zero and since $\mathrm{vol}(Q_i^*)=2^{-\nu_i d}$,
\[
\|f\|_{F^s_{p\tau}(\IR^d)}^p
\lesssim \sum_{i=1}^N\int_{Q_i^*}(2^{-\nu_i})^{\beta p-d}\mathbf{1}_{Q_i^*}(x){\rm d} x
= \sum_{i=1}^N(2^{-\nu_i})^{\beta p}.
\]
As $\beta p=\gamma +d$ and $2^{-\nu_i}\asymp d_i$, we arrive at
\[
\|f\|_{F^s_{p\tau}(\Omega)}^p
\lesssim \sum_{i=1}^N d_i^{\gamma+d}.
\]
Using a substitution we have	
\[
\|f\|_{L_q(\Omega)}^q
=\sum_{i=1}^N\int_{Q_i^*}(2^{-\nu_i})^{\beta q}\psi_i(x)^q{\rm d} x
=\sum_{i=1}^N(2^{-\nu_i})^{\beta q+sq-d q/p+d}\int_{2Q}\psi(x)^q{\rm d} x.
\]
Since $\beta q+sq = \gamma +d q/p$ we have
\[
\|f\|_{L_q(\Omega)}^q
\asymp\sum_{i=1}^N d_i^{\gamma+d}.
\]
It follows that the normalized function $f_*=f/\|f\|_{F^s_{p\tau}(\Omega)}$	satisfies
\[
\|f_*\|_{L_q(\Omega)}
=\frac{\|f\|_{L_q(\Omega)}}{\|f\|_{F^s_{p\tau}(\Omega)}}
\gtrsim \Big(\sum_{i=1}^N d_i^{\gamma+d}\Big)^{1/q-1/p}
\gtrsim \Big(\int_{\Omega}\dist(x,P)^{\gamma}{\rm d} x \Big)^{s/\gamma},
\]
where the last inequality derives from~\eqref{eq:LB-discretisation}. Now the fooling function $f_\ast$ yields the lower bound analogous to~\eqref{eq:fooling}. 
We get the lower bound for the integration problem as the function $f_*$ is non-negative and thus $\|f\|_{L_1(\Omega)}=\int_{\Omega} f(x){\rm d} x$.
\end{proof}

Let us conclude with a few remarks on further extensions of our result.

\begin{remark}[More general domains, part~2]\label{rem:domains2}
 We only used that $\Omega$ is a bounded Lipschitz domain in order to extend functions on $\Omega$ to functions on $\IR^d$ without significant change of the norm. It is not hard to check that our results are valid for all bounded measurable sets $\Omega$ satisfying a strong local cone condition as in Lemma~\ref{lem:local-CC} if the function spaces on $\Omega$ are defined via restriction of the corresponding function spaces on $\IR^d$, i.e.,
 \[
 A(\Omega):=\{ f\vert_\Omega \colon f\in A(\IR^d)\},
 \qquad \Vert g \Vert_{A(\Omega)} := \inf_{f \in A(\IR^d)\colon f\vert_\Omega=g} \Vert f \Vert_{A(\IR^d)}
 \] 
 for $A\in \{W_p^s, F^s_{p\tau}, C^s\}$.
 By Lemma~\ref{lem:extension-sobolev} (and its analogues), these spaces coincide 
 with the spaces from above
 (in the sense of equivalent norms) if $\Omega$ is a bounded Lipschitz domain.
\end{remark}

\begin{remark}[Besov spaces]\label{rem:Besov}
It remains open whether Theorem \ref{thm:main-result-extended} carries over also to the scale of Besov spaces $ B^s_{p\tau}(\Omega), $ where $ 0<p,\tau\le \infty $ and $ s>d/p $. The upper bounds are complicated by finding an analogue of \eqref{eq:norms-efficient-fspace}. 
Novak and Triebel \cite{NT06} avoid this employing interpolation, which may also be used here 
in order to treat the case $  q\ge p $, but for $ q<p $ different methods may be required. 
For the lower bound the situation is similar to the $ F $-spaces. We leave the extension of Theorem \ref{thm:main-result-extended} to Besov spaces for future research.
\end{remark}

\begin{remark}[Functions with zero boundary condition]\label{rem:zero-boundary}
Since all our fooling functions are supported in the open set $\Omega$,
all our results hold without changes for the smaller spaces of functions
satisfying zero boundary conditions.
\end{remark}

\paragraph{Acknowledgement.} 
We thank Simon Hackl for bringing Besicovitch-type covering results to our attention, Erich Novak for fruitful discussions about polynomial reproduction
and both Winfried Sickel and Mario Ullrich for pointing us toward characterization of $ F $-spaces in terms of wavelets and atoms.  The authors were supported by the Austrian Science Fund (FWF) through projects F5513-N26 (DK, MS), which is part of the Special Research Program ``Quasi-Monte Carlo Methods: Theory and Applications'', and P32405 ``Asymptotic Geometric Analysis and Applications'' (MS).

\begin{small}
\bibliographystyle{plain}
\bibliography{raninfo}

\begin{thebibliography}{10}

\bibitem{AF03}
{R.\,A.} Adams and {J.\,J.\,F.} Fournier.
\newblock {\em Sobolev spaces}, volume 140 of {\em Pure and Applied
  Mathematics}.
\newblock Elsevier/Academic Press, Amsterdam, second edition, 2003.

\bibitem{ALdST07}
R.~Arcang\'{e}li, {M.\,C.} L\'{o}pez~de Silanes, and {J.\,J.} Torrens.
\newblock An extension of a bound for functions in {S}obolev spaces, with
  applications to {$(m,s)$}-spline interpolation and smoothing.
\newblock {\em Numer. Math.}, 107(2):181--211, 2007.

\bibitem{A50}
N.~Aronszajn.
\newblock Theory of reproducing kernels.
\newblock {\em Trans. Amer. Math. Soc.}, 68:337--404, 1950.

\bibitem{BDS+15}
{J.\,S.} Brauchart, J.~Dick, {E.\,B.} Saff, {I.\,H.} Sloan, {Y.\,G.} Wang, and
  {R.\,S.} Womersley.
\newblock Covering of spheres by spherical caps and worst-case error for equal
  weight cubature in {S}obolev spaces.
\newblock {\em J. Math. Anal. Appl.}, 431(2):782--811, 2015.

\bibitem{BS08}
S.~Brenner and L.~Scott.
\newblock {\em {The Mathematical Theory of Finite Element Methods}}, volume~15
  of {\em Texts in Applied Mathematics}.
\newblock Springer, 2008.

\bibitem{Cia78}
{P.\,G.} Ciarlet.
\newblock {\em {The finite element method for elliptic problems}}.
\newblock Cambridge University Press, Amsterdam, North-Holland, 1978.

\bibitem{Coh04}
P.~Cohort.
\newblock Limit theorems for random normalized distortion.
\newblock {\em Ann. Appl. Probab.}, 14(1):118--143, 2004.

\bibitem{DeG75}
M.~de~Guzm\'{a}n.
\newblock {\em Differentiation of integrals in {$R^{n}$}}.
\newblock Springer-Verlag, Berlin-New York, 1975.

\bibitem{DL04}
S.~Dekel and D.~Leviatan.
\newblock Whitney estimates for convex domains with applications to
  multivariate piecewise polynomial approximation.
\newblock {\em Found.\ Comput.\ Math.}, 4(4):345--368, 2004.

\bibitem{DS93}
{R.\,A.} DeVore and {R.\,C.} Sharpley.
\newblock Besov spaces on domains in {$\Bbb R^d$}.
\newblock {\em Trans. Amer. Math. Soc.}, 335(2):843--864, 1993.

\bibitem{Duc78}
J.~Duchon.
\newblock Sur l'erreur d'interpolation des fonctions de plusieurs variables par
  les {$D^{m}$}-splines.
\newblock {\em RAIRO Anal. Num\'{e}r.}, 12(4):325--334, vi, 1978.

\bibitem{DS80}
T.~Dupont and R.~Scott.
\newblock {Polynomial Approximation of Functions in Sobolev Spaces}.
\newblock {\em Math. Comp.}, 34(150):441--463, 1980.

\bibitem{ET96}
{D.\,E.} Edmunds and H.~Triebel.
\newblock {\em Function spaces, entropy numbers, differential operators},
  volume 120 of {\em Cambridge Tracts in Mathematics}.
\newblock Cambridge University Press, Cambridge, 1996.

\bibitem{EGO19}
M.~Ehler, M.~Graef, and {C.\,J.} Oates.
\newblock {Optimal Monte Carlo integration on closed manifolds}.
\newblock {\em Statistics and Computing}, 29:1203--1214, 2019.

\bibitem{GL00}
S.~Graf and H.~Luschgy.
\newblock {\em Foundations of quantization for probability distributions},
  volume 1730 of {\em Lecture Notes in Mathematics}.
\newblock Springer-Verlag, Berlin, 2000.

\bibitem{Hei94}
S.~Heinrich.
\newblock Random approximation in numerical analysis.
\newblock In K.~D. Bierstedt and et~al., editors, {\em Functional Analysis},
  pages 123--171. Dekker, New York, 1994.

\bibitem{HKNPUsurvey}
A.~Hinrichs, D.~Krieg, E.~Novak, J.~Prochno, and M.~Ullrich.
\newblock {On the power of random information}.
\newblock In F.~J. Hickernell and P.~Kritzer, editors, {\em Multivariate
  Algorithms and Information-Based Complexity}, pages 43--64. De Gruyter,
  Berlin/Boston, 2020.

\bibitem{HKNPUellipsoids}
A.~Hinrichs, D.~Krieg, E.~Novak, J.~Prochno, and M.~Ullrich.
\newblock {Random sections of ellipsoids and the power of random information}.
\newblock {\em Trans. Amer. Math. Soc.}, 374:8691--8713, 2021.

\bibitem{HKNV20}
A.~Hinrichs, D.~Krieg, E.~Novak, and J.~Vyb\'iral.
\newblock {Lower bounds for the error of quadrature formulas for Hilbert
  spaces}.
\newblock {\em J. Complex.}, 65:101544, 2021.

\bibitem{KNS}
D.~Krieg, E.~Novak, and S.~Sonnleitner.
\newblock {Recovery of Sobolev functions restricted to iid sampling}.
\newblock {\em Math. Comp.}, https://doi.org/10.1090/mcom/3763, 2022.

\bibitem{KSmani}
D.~Krieg and S.~Sonnleitner.
\newblock {Function recovery on manifolds using scattered data}.
\newblock {\em arXiv preprints}, arXiv:2109.04106, 2021.

\bibitem{KU19}
D.~Krieg and M.~Ullrich.
\newblock {Function values are enough for $L_2$-approximation}.
\newblock {\em Found. Compu. Math.}, 21:1141--1151, 2021.

\bibitem{M85}
{V.\,G.} Maz'ya.
\newblock {\em Sobolev spaces}.
\newblock Springer Series in Soviet Mathematics. Springer-Verlag, Berlin, 1985.
\newblock Translated from the Russian by T. O. Shaposhnikova.

\bibitem{Mha10}
{H.\,N.} Mhaskar.
\newblock Eignets for function approximation on manifolds.
\newblock {\em Appl. Comput. Harmon. Anal.}, 29(1):63--87, 2010.

\bibitem{NWW04}
{F.\,J.} Narcowich, {J.\,D.} Ward, and H.~Wendland.
\newblock {Sobolev bounds on functions with scattered zeros, with applications
  to radial basis function surface fitting}.
\newblock {\em Math. Comp.}, 74(250):743--763, 2004.

\bibitem{NT06}
E.~Novak and H.~Triebel.
\newblock {Function spaces in Lipschitz domains and optimal rates of
  convergence for sampling}.
\newblock {\em Constr. Approx.}, 23:325--350, 2006.

\bibitem{NUWZ18}
E.~Novak, M.~Ullrich, H.~Wo\'{z}niakowski, and S.~Zhang.
\newblock Reproducing kernels of {S}obolev spaces on {$\Bbb R^d$} and
  applications to embedding constants and tractability.
\newblock {\em Anal. Appl. (Singap.)}, 16(5):693--715, 2018.

\bibitem{NW08}
E.~Novak and H.~Wo\'{z}niakowski.
\newblock {\em Tractability of multivariate problems. {V}ol. 1: {L}inear
  information}, volume~6 of {\em EMS Tracts in Mathematics}.
\newblock European Mathematical Society (EMS), Z\"{u}rich, 2008.

\bibitem{NW10}
E.~Novak and H.~Wo\'{z}niakowski.
\newblock {\em Tractability of multivariate problems. {V}olume {II}: {S}tandard
  information for functionals}, volume~12 of {\em EMS Tracts in Mathematics}.
\newblock European Mathematical Society (EMS), Z\"{u}rich, 2010.

\bibitem{P98}
G.~Pag\`es.
\newblock A space quantization method for numerical integration.
\newblock {\em J. Comput. Appl. Math.}, 89(1):1--38, 1998.

\bibitem{RS16}
A.~Reznikov and {E.\,B.} Saff.
\newblock The covering radius of randomly distributed points on a manifold.
\newblock {\em Int. Math. Res. Not. IMRN}, 2016(19):6065--6094, 2016.

\bibitem{Rych99}
{V.\,S.} Rychkov.
\newblock On restrictions and extensions of the {B}esov and
  {T}riebel-{L}izorkin spaces with respect to {L}ipschitz domains.
\newblock {\em J. London Math. Soc. (2)}, 60(1):237--257, 1999.

\bibitem{Ste71}
{E.\,M.} Stein.
\newblock {\em Singular integrals and differentiability properties of
  functions}.
\newblock Princeton University Press, Princeton, New Jersey, 1971.

\bibitem{S79}
{A.\,G.} Suharev.
\newblock Optimal formulas of numerical integration for some classes of
  functions of several variables.
\newblock {\em Dokl. Akad. Nauk SSSR}, 246(2):282--285, 1979.

\bibitem{ST-J80}
S.~Szarek and N.~Tomczak-Jaegermann.
\newblock On nearly {E}uclidean decomposition for some classes of {B}anach
  spaces.
\newblock {\em Compositio Math.}, 40(3):367--385, 1980.

\bibitem{Tri11}
H.~Triebel.
\newblock {\em Fractals and spectra}.
\newblock Birkh\"{a}user Verlag, Basel, 2011.

\bibitem{Ull20}
M.~Ullrich.
\newblock {On the worst-case error of least squares algorithms for
  $L_2$-approximation with high probability}.
\newblock {\em J.\ Complex.}, 60, 2020.

\bibitem{WSU72}
Wayne~State University.
\newblock {Classroom {N}otes: Every convex function is locally Lipschitz}.
\newblock {\em Amer. Math. Monthly}, 79:1121--1124, 1972.

\bibitem{W04}
H.~Wendland.
\newblock {\em Scattered data approximation}, volume~17 of {\em Cambridge
  Monographs on Applied and Computational Mathematics}.
\newblock Cambridge University Press, Cambridge, 2005.

\bibitem{WS93}
{Z.\,M.} Wu and R.~Schaback.
\newblock Local error estimates for radial basis function interpolation of
  scattered data.
\newblock {\em IMA J. Numer. Anal.}, 13(1):13--27, 1993.

\bibitem{Yuk08}
{J.\,E.} Yukich.
\newblock Limit theorems for multi-dimensional random quantizers.
\newblock {\em Electron. Commun. Probab.}, 13:507--517, 2008.

\end{thebibliography}
\end{small}

\end{document}